\documentclass[11pt]{article}
\usepackage[utf8]{inputenc}
\usepackage{amssymb,amsmath,amsthm}
\usepackage[all,cmtip]{xy}
\usepackage{hyperref}
\usepackage{enumerate}
\usepackage{mathtools}
\author{Anssi Lahtinen, David Sprehn}

\newcommand{\acomment}[1]{} 

\newcommand{\isom}{\approx}
\newcommand{\homot}{\simeq}

\newcommand{\tensor}{\otimes}

\newcommand{\pt}{\mathrm{pt}}
\newcommand{\modmod}{{\slash\!\!\slash}} 
\providecommand{\sslash}{}
\renewcommand{\sslash}{\modmod}
\newcommand{\Emb}{\mathrm{Emb}}
\newcommand{\id}{\mathrm{id}}

\newcommand{\Mat}{\mathrm{Mat}}

\makeatletter

\newcommand\xto[2][]{%
  \ext@arrow 1579{\rightarrowfill@}{#1}{#2}}

\newcommand\xot[2][]{%
  \ext@arrow 5197{\leftarrowfill@}{#1}{#2}}

\providecommand*{\twoheadrightarrowfill@}{%
  \arrowfill@\relbar\relbar\twoheadrightarrow
}

\newcommand\xtwoheadrightarrow[2][]{%
  \ext@arrow 1579\twoheadrightarrowfill@{#1}{#2}%
}

\makeatother

\newcommand{\incl}{\hookrightarrow}

\newcommand{\longto}{\longrightarrow}

\newcommand{\im} {\operatorname{Im}}
\renewcommand{\hom} {\operatorname{Hom}}
\newcommand{\sym} {\operatorname{Sym}}
\newcommand{\aut} {\mathrm{Aut}}

\newcommand{\iso} {\approx}

\newcommand{\trna}{\mathrm{tr}}

\newcommand{\inv} {\ensuremath ^{-1}}
\newcommand{\F} {\ensuremath{\mathbb{F}}}
\newcommand{\Z} {\ensuremath{\mathbb{Z}}}

\newcommand{\E} {\ensuremath{\mathbb{E}}}

\newcommand{\A} {\ensuremath{\mathbb{A}}}

\newcommand{\set}[2] {\left\lbrace {#1} \,\middle\arrowvert\, {#2} \right\rbrace}

\newcommand{\+} {\quad\mbox{and}\quad}
\newcommand{\Aff} {\mathrm{Aff}}

\DeclareMathOperator\Fr{Fr}

\DeclareMathOperator\Gr{Gr}

\newtheorem{thm}{Theorem}
\newtheorem{prop}[thm]{Proposition}
\newtheorem{lemma}[thm]{Lemma}
\newtheorem{cor}[thm]{Corollary}
\theoremstyle{definition}
\newtheorem{remark}[thm]{Remark}
\newtheorem{vista}[thm]{Vista}

\newtheorem{defn}[thm]{Definition}

\newcommand{\bbN}{\ensuremath{\mathbb{N}}}

\newcommand{\bbP}{\ensuremath{\mathbb{P}}}

\newcommand{\bbZ}{\ensuremath{\mathbb{Z}}}

\begin{document}

\title{\sc Modular characteristic classes for representations over finite fields}
\date{\today}
\author{Anssi Lahtinen, David Sprehn}

\maketitle

\begin{abstract}
The cohomology of the degree-$n$ general linear group
over a finite field of characteristic $p$,
with coefficients also in characteristic $p$,
remains poorly understood.
For example, the lowest degree previously known to contain nontrivial elements
is exponential in $n$.
In this paper, we introduce a new system of characteristic 
classes for representations over finite fields, and 
use it to construct a wealth of explicit nontrivial elements
in these cohomology groups.
In particular we obtain nontrivial elements in degrees linear in $n$.
We also construct nontrivial elements in the mod $p$
homology and cohomology 
of the automorphism groups of free groups, and the general linear 
groups over the integers. These elements reside in the 
unstable range where the homology and cohomology remain 
mysterious.
\end{abstract}


\section{Introduction}
We introduce a new system of 
modular characteristic classes for 
representations of groups over finite fields, and use it to construct
explicit non-trivial elements in the modular cohomology of the general 
linear groups over finite fields.
The cohomology groups $H^\ast(GL_N\F_{p^r};\, \F)$
were computed by Quillen~\cite{QuillenK} in the case where
$\F$ is a field of characteristic different from $p$,
but he remarked that determining them in the modular
case where the characteristic of $\F$ is $p$
``seems to be a difficult problem once $N\geq3$''
\cite[p.~578]{QuillenK}.
Indeed, 
the modular cohomology has since resisted computation for four decades.
Complete calculations exist only for $N\leq 4$~\cite{Aguade,TYGL3,TYGL4,adem1990symmetric}.  
Much attention has focused on the case where $N$ is small compared to $p$,
e.g.~\cite{Barbu,BNP,BNP2,LieType}.  

To our knowledge, when $N>\max\{p,4\}$, the only previously constructed
nonzero elements of $H^*(GL_N\F_{p^r};\F_p)$
are those due to Milgram and Priddy \cite{MilgramPriddy}, in the case $r=1$.
These reside in 
exponentially high degree: at least $p^{N-2}$.
On the other hand, the cohomology is known to vanish in 
degrees less than $N/2$,
by the stability theorem of 
Maazen~\cite{Maazen} 
together with Quillen's computation~\cite{QuillenK}
that the stable limit is zero.
This leaves a large degree gap where it was not known whether 
the cohomology groups are nontrivial.
We narrow this gap considerably
by providing nonzero classes in degrees linear in $N$.
We obtain:
\begin{thm}[see~Theorem~\ref{nonvanishing}]
\label{introthm}
Let $N\geq2$, and let $n$ be the natural number satisfying
\[ p^{n-1}<N\leq p^n. \]
Then 
\[ H^\ast(GL_N\F_{p^r};\F_p) \]
has a nonzero element in degree 
$r(2p^n-2p^{n-1}-1)$.
Moreover, it has a non-nilpotent element in degree
$2r(p^n-1)$
if $p$ is odd and in degree $r(2^n-1)$ if $p=2$.
\qed
\end{thm}
\noindent
Notice that the degrees in the theorem grow linearly with $N$: 
if $d$ is any of the degrees mentioned in the theorem, then
\[ r(N-1)\leq d< 2r(pN-1). \]
In the case $r=1$, we obtain
stronger results, for instance:
\begin{thm}[see~Theorem~\ref{r1nonvanishing}]
\label{introthm2}
For all $N \geq 2$,
\[ H^\ast(GL_N\F_2;\F_2) \]
has a non-nilpotent element of degree $d$ for every 
$d$ with at least $\lceil\log_2 N\rceil$ 
ones in its binary expansion. \qed
\end{thm}

Our characteristic classes are defined
for representations of dimension $N\geq2$ over the finite field $\F_{p^r}$,
and they are modular in the sense that they take values in group 
cohomology with 
coefficients in a field $\F$ of characteristic $p$.
Thus they are interesting even for $p$-groups.
The family of characteristic classes is parametrized by the 
cohomology of $GL_2\F_{p^r}$.
We will show that many classes in this family are nonzero by finding
representations $\rho$ on which they are nontrivial.  This will produce a family of nonzero 
cohomology classes on the general linear groups, namely the ``universal classes'' 
obtained by applying the characteristic classes 
to the defining representation of $GL_N\F_{p^r}$ where $N$ is the 
dimension of $\rho$.

The characteristic classes are defined in terms of a push--pull
construction featuring a transfer map.
This construction was previously studied by the second author
in \cite{LieType}, where he proved that it yields an injective map
\[ H^*(GL_2\F_{p^r};\F_p)\longto H^*(GL_N\F_{p^r};\F_p) \]
for $2\leq N\leq p$.
The present work was inspired by computations of the 
first author in string topology of classifying spaces 
\cite{StringTopComp} featuring similar push--pull constructions.

In addition to the groups $GL_N\F_{p^r}$, 
our characteristic classes can be used to study other groups 
with interesting representations over finite fields.
For example:
\begin{thm}[see~Theorem~\ref{thm:glzautfn} and Proposition~\ref{primecasenonvanishing}]
\label{thm:intro-othergrps}
For all $n\geq 1$,
\[ 
	H^*(\aut(F_{p^n});\,\F_p)
	\qquad\text{and}\qquad
	H^*(GL_{p^n}\Z;\,\F_p)	
\]
have a non-nilpotent element of degree $2d$ for every $d$ 
whose $p$-ary digits sum up to $k(p-1)$ for some $k \geq n$.
In particular, there is a non-nilpotent element of degree $2p^n-2$.
(For $p=2$, divide degrees by 2.) \qed
\end{thm}
\noindent
These classes live in the unstable range where 
the cohomology groups remain poorly understood.
For previous work on unstable classes in the homology and cohomology 
of $\aut(F_n)$ and $GL_n(\Z)$ (as well as the closely related
groups $\mathrm{Out}(F_n)$ and $SL_n(\Z)$), see for example
\cite{Chen,MR1709957,MR1671188,MR1734418,MR1781542,MR1815782,MR2119302,MR2367210,
MR2580660,Gray,MR3029423,CHKV}
and
\cite{MR0491893,MR0470141,MR1029906,MR1167176,MR1332169,MR3084439,MR3290086}.

The paper proceeds as follows.  In section~\ref{constructions},
we summarize the behavior 
of our characteristic classes and give two equivalent constructions of them.
In section~\ref{parabolicinduction}, we show that the universal classes on $GL_N\F_{p^r}$ for various $N$ are related by parabolic induction maps.
In section~\ref{decomposables}, we prove that our characteristic classes
 vanish on representations decomposable as a direct sum.  In sections~\ref{cohofGL2} and~\ref{coprod}, we recall the cohomology of $GL_2\F_{p^r}$ and describe a 
coalgebra structure on it.  In section~\ref{wedge}, we develop a formula for the characteristic classes of a representation decomposable as a ``wedge sum.''  
In section~\ref{basicrep}, we introduce a family of 
``basic representations'' which decompose as such wedge sums,
and in section~\ref{nontriviality}, we
show that many of the characteristic classes of these
representations are nonzero.  
In section~\ref{elemab}, we study representations of elementary abelian groups, proving that their characteristic classes agree with 
those of a certain subrepresentation. 
In section~\ref{largerep}, we exploit this property to construct
high-dimensional representations with the same characteristic 
classes as the basic
representations, and deduce our main results on the cohomology of general linear groups (Theorems~\ref{nonvanishing} and \ref{r1nonvanishing}).  
In section~\ref{Dickson}, we restrict to the prime fields 
$\F_p$, and give an alternative description of
some of the characteristic classes of the basic representations 
in terms of Dickson invariants.
This allows us to deduce that a certain subset of the 
universal classes is algebraically independent
(Theorem~\ref{thm:algindependence}).
Finally, in section~\ref{sec:othergroups},
we present applications to the homology and cohomology
of automorphism groups of free groups and 
general linear groups over the integers.

\subsection*{Conventions}

Throughout the paper, $p$ is a prime and $q = p^r$ is a power of $p$.
A ``representation'' means a finite-dimensional 
representation over the field $\F_q$.
Unless noted otherwise, 
homology and cohomology will be with coefficients in a field $\F$
of characteristic $p$, which will be omitted from notation.
The main interest is in the cases $\F = \F_p$ and $\F = \F_q$,
the latter being interesting since it allows for an explicit 
description of $H^\ast(GL_2\F_q)$.

As our characteristic classes $\chi_\alpha$ 
are only defined for representations of dimension at least 2,
whenever the notation $\chi_\alpha(\rho)$ appears, it is implicitly
assumed that the representation $\rho$ has dimension at least 2.

\section{Characteristic classes}\label{constructions}

We will now define our characteristic classes.
Given $\alpha\in H^\ast(GL_2\F_q)$,
to each $N$-dimensional representation $\rho$ 
over $\F_q$ of a group $G$ (with $N\geq 2$), we will associate
a class 
\[
	\chi_\alpha(\rho) \in H^\ast(G).
\]
only depending on the isomorphism class of $\rho$.
As required of a characteristic class, 
the assignment $\rho \mapsto \chi_\alpha(\rho)$
will be natural with respect to group homomorphisms
in the sense that 
\[
	\chi_\alpha(f^\ast\rho) = f^\ast\chi_\alpha(\rho)
\]
for any group homomorphism $f$ to the domain of $\rho$.
In fact we will give two alternative constructions of the 
characteristic classes,
Definitions~\ref{def:chiclassdef1} and \ref{def:chiclassdef2} 
below, and prove their 
equivalence. The following theorem summarizes 
the basic properties of the classes.
\begin{thm}
\label{thm:chiclassprops} 
The characteristic classes $\chi_\alpha$
have the following properties.
\begin{enumerate}[(i)]
\item (Nontriviality, Remark~\ref{rk:quickprops})
\label{thmpart:props-nontriviality}
For the identity representation of $GL_2\F_q$, we have
\[ \chi_\alpha(\id_{GL_2\F_q})=\alpha \] 
for all $\alpha \in H^{>0}(GL_2\F_q)$.
\item (Vanishing on decomposables, Corollary~\ref{cor:vanishingondecomposables})
\label{thmpart:props-vanishinondecomposables}
If $\rho$ and $\eta$ are nonzero representations of $G$, then
\[ \chi_\alpha(\rho\oplus\eta)=0. \]
\item (Wedge sum formula, Theorem~\ref{wedgeformula}) 
	\label{thmpart:props-wedgesum}
	Suppose $\rho$ and $\eta$ are representations of $G$,
	and $v_0\in \rho$ and $w_0 \in \eta$ are vectors 
	fixed by $G$. Then
	\[ 
    	\chi_\alpha(\rho\vee\eta)
    	=
    	-\sum \chi_{\alpha_{(1)}}(\rho)\cdot \chi_{\alpha_{(2)}}(\eta)
    	, 
    \]
    where
    $\rho \vee \eta = \rho\oplus\eta / \langle v_0-w_0 \rangle$ 
	and
	\[ \Delta(\alpha)=\sum \alpha_{(1)}\otimes\alpha_{(2)} \]
    is a coproduct on the cohomology of $GL_2\F_q$
    which will be described in section~\ref{coprod}.
\item (Reduction to $J_1$, Theorem~\ref{J1})
\label{thmpart:props-J1}
	If $\rho$ is a representation of an elementary abelian
	$p$-group $G$, then 
	\[
		\chi_\alpha(\rho) = \chi_\alpha(J_1(\rho)),
	\]
	where $J_1(\rho) \subset \rho$ is the subrepresentation 
	consisting of the vectors annihilated by the second 
	power of the augmentation ideal	in the group ring $\F_q G$.
\item (Dependence on $\alpha$, Remark~\ref{rk:quickprops}) 
\label{thmpart:props-dependenceonalpha}
	For a fixed representation $\rho$ of a group $G$, the 
	map 
	\[
		H^\ast(GL_2\F_q;\,\F) \longto H^\ast(G;\,\F),
		\quad
		\alpha \longmapsto \chi_\alpha(\rho)
	\]
	is $\F$-linear and commutes with the action of the 
	mod p Steenrod algebra and the operation $\Fr_*$ induced by the
	Frobenius map of the coefficient field $\F$. Moreover,
	it commutes with coefficient field extension
\[ H^\ast(-;\F) \longto H^\ast(-;\E)=H^\ast(-;\F)\otimes_\F \E, \]
i.e. $\chi_{\alpha\otimes1}(\rho)=\chi_\alpha(\rho)\otimes1$.
\end{enumerate}    
\end{thm}
\noindent
In part~(\ref{thmpart:props-dependenceonalpha}), the mod $p$ Steenrod
algebra acts on cohomology with $\F$-coefficients by extension
of scalars from cohomology with $\F_p$-coefficients.
Notice that part (\ref{thmpart:props-vanishinondecomposables})
suggests a way to detect indecomposable representations: 
if $\chi_\alpha(\rho)\neq 0$
for some $\alpha$, then $\rho$ is indecomposable.

We now give our two definitions of the $\chi_\alpha$-classes.
\begin{defn}
\label{def:chiclassdef1}
For $\alpha\in H^{>0}(GL_2\F_q)$ and $N \geq 2$, let 
\[
	\chi_\alpha^{(N)} =  (i_!\circ\pi^\ast)(\alpha)
\]
in $H^d(GL_N\F_q)$, where $i$ denotes the inclusion of
the parabolic subgroup
\begin{equation}
\label{eq:pdef}
	P 
	= 	
    \begin{bmatrix}
	GL_2 
	& 
	* 
	\\ 
	& 
	GL_{N-2}
    \end{bmatrix}
    \leq 
    GL_N\F_q
 \end{equation}
into $GL_N\F_q$,
$\pi\colon P \to GL_2\F_q$ is the projection onto
the diagonal copy of $GL_2\F_q$, and $i_!$ denotes the transfer map 
induced by $i$.
By convention, set $\chi_\alpha^{(N)}=0$ when $\deg\alpha=0$.  
(This convention simplifies the statement of Theorem~\ref{wedgeformula}).

Now let $G$ be a group
and let $\rho$ be a representation of $G$ of dimension
$N \geq 2$.
We define
\[
	\chi_\alpha(\rho) = \rho^\ast (\chi^{(N)}_\alpha) \in H^* (G).
\]
On the right side, the notation $\rho$ denotes the homomorphism
$G \to GL_N\F_q$ associated to the representation, well defined up to conjugacy.
It is evident from the definition that $\chi_\alpha(\rho)$ 
only depends on the isomorphism class of $\rho$ and 
satisfies the required naturality.
We call the underlying class 
$\chi_\alpha^{(N)} = \chi_\alpha(\id_{GL_N\F_q}) \in H^d(GL_N\F_q)$
the \emph{universal $\chi_\alpha$-class for $N$-dimensional
representations}.
\end{defn}

Now we give the second definition.
If $V$ and $W$ are vector spaces, write $\Emb(V,W)$ for the set 
of linear embeddings of $V$ into $W$.
If $X$ is a $G$-space or $G$-set,
we write $X\sslash G$ for the 
topological bar construction 
$B(\pt,G,X)$ \cite[Section 7]{MayClassifying}, a model 
for the homotopy orbit space $EG \times_G X$.
In particular, $\pt\sslash G$ is a model for the classifying 
space $BG$.
\begin{defn}
\label{def:chiclassdef2}
Let $G$ be a group and let $\rho$ be a representation of $G$ 
of dimension $N \geq 2$.
Using the diagram
\begin{equation}
\label{diag:chiclassdef2}
\vcenter{\xymatrix@!0@C=8em@R=8ex{
	&
	\Emb(\F_q^2,\rho)\sslash G\times GL_2\F_q
	\ar[dl]^{\pi}_{!}
	\ar[dr]_{\tau}
	\\
	\pt\sslash G
	&&
	\pt\sslash GL_2\F_q
}}
\end{equation}
define
\[
	\chi_\alpha(\rho) = (\pi^{}_!\circ\tau^\ast)(\alpha) \in H^*(G)
\]
for $\alpha\in H^{>0} GL_2\F_q$.
Here $G \times GL_2\F_q$ acts on $\Emb(\F_q^2,\rho)$ by
\begin{equation}
\label{eq:gxgl2action} 
	(g,A) \cdot f = g\circ f  \circ A^{-1},
\end{equation}
and $\pi$ and $\tau$ are the evident projection maps.
Observe that the action of the $GL_2\F_q$-factor on 
$\Emb(\F_q^2,\rho)$ is free. Therefore, 
the map $\pi$ factors as the composite
\[
	\Emb(\F_q^2,\rho)\sslash G\times GL_2\F_q
	\xto{\ \homot\ }
	(\Emb(\F_q^2,\rho)/GL_2\F_q) \sslash G
	\longto
	\pt \sslash G,
\]
of a homotopy equivalence and 
a covering space with fibres modeled on 
the Grassmannian $\Gr_2(\rho)$, a finite set. Thus $\pi$ indeed 
admits a transfer map $\pi_!$
on cohomology.
As before, set $\chi_\alpha(\rho)=0$ if $\dim\alpha=0$.
Clearly $\chi_\alpha(\rho)$ only depends on the isomorphism
class of $\rho$, and compatibility of transfer maps with pullbacks 
implies that $\chi_\alpha$ satisfies the required naturality.
\end{defn}

\begin{prop}
\label{prop:chiclassdefequiv}
The classes $\chi_\alpha(\rho)$ of
Definitions~\ref{def:chiclassdef1} and \ref{def:chiclassdef2}
agree.
\end{prop}
\begin{proof}
Write temporarily $\tilde{\chi}_\alpha$ for the the classes
given by Definition~\ref{def:chiclassdef2}.
By naturality, it suffices to show that 
$\tilde{\chi}_\alpha(\id_{GL_N\F_q}) = \chi_\alpha^{(N)}$ 
for all $N\geq 2$. Observe that for $\rho=\id_{GL_N\F_q}$,
the $GL_N\F_q \times GL_2\F_q$-action on 
$\Emb(\F_q^2,\F_q^N)$ given by 
\eqref{eq:gxgl2action}
is transitive, with the  stabilizer
of the inclusion $\F_q^2 \incl \F_q^N$, $v\mapsto (v,0)$ given by
the subgroup
\[
	\left\{
		\left(
        	\begin{bmatrix}	 
        	A_{11} & A_{12}\\
        	0 & A_{22}
        	\end{bmatrix},			
			A_{11}
		\right)
		\in
		GL_N\F_q \times GL_2\F_q
    	\,\Bigg|\,
        \begin{matrix}
       	   	A_{11} \in GL_2\F_q \\
           	A_{22} \in GL_{N-2}\F_q \\
			A_{12} \in \Mat_{N,N-2}\F_q
        \end{matrix}
	\right\}	
\]
of $GL_N\F_q \times GL_2\F_q$.
Thus diagram \eqref{diag:chiclassdef2} for computing 
$\tilde{\chi}_\alpha(\id_{GL_N\F_q})$
reduces to the diagram
\[\xymatrix@!0@C=8em@R=8ex{
	&
	\pt\sslash P
	\ar[dl]_{!}
	\ar[dr]
	\\
	\pt\sslash GL_N\F_q
	&&
	\pt\sslash GL_2\F_q
}\]
where $P$ is the parabolic subgroup
\eqref{eq:pdef} of Definition~\ref{def:chiclassdef1}
and the two maps are induced by the inclusion 
$i\colon P \incl GL_N\F_q$
and the projection $\pi\colon P \to GL_2\F_q$ of 
Definition~\ref{def:chiclassdef1}. The claim follows.
\end{proof}

\begin{remark}
\label{rk:quickprops}
In the case $N=2$,
the maps $i$ and $\pi$ of Definition~\ref{def:chiclassdef1}
both reduce to the identity map of $GL_2\F_q$. 
Thus part~(\ref{thmpart:props-nontriviality}) of 
Theorem~\ref{thm:chiclassprops}
follows.
Part~(\ref{thmpart:props-dependenceonalpha})
of 
Theorem~\ref{thm:chiclassprops}
is immediate from the fact that induced maps and 
transfer maps on cohomology have the properties in question.  (Indeed, transfer maps in group cohomology commute with the Steenrod powers~\cite{Evens}, and with maps induced by coefficient group homomorphisms.)
\end{remark}

\section{Parabolic induction}\label{parabolicinduction}
The aim of this section is to show that the universal 
$\chi_\alpha$-classes $\chi_\alpha^{(n)}$
for various $n$ are related by what we call
\emph{parabolic induction maps}.
\begin{defn}
\label{def:parindmaps}
For $m \leq n$, we define
the \emph{parabolic induction map} $\Phi_{m,n}$
to be the composite
\[
	\Phi_{m,n}
	\colon 
	H^\ast(GL_m\F_q) 
	\xto{\ \pi^\ast\,} 
	H^\ast(P)
	\xto{\ i_!\ }
	H^\ast(GL_n\F_q)
\]
where $P$ is the parabolic subgroup
\[
	P
	=
   	\begin{bmatrix}	 
    	GL_m \F_q & \ast \\
    	0 & GL_{n-m}\F_q
	\end{bmatrix}
	\leq
	GL_n\F_q,
\]
$i\colon P \incl GL_n\F_q$ is the inclusion, and
$\pi \colon P \to GL_m\F_q$ is the projection onto the 
diagonal copy of $GL_m\F_q$.
\end{defn}

\begin{remark}
\label{rk:chiparindconn}
Comparing Definitions~\ref{def:chiclassdef1} and 
\ref{def:parindmaps}, we see that 
$\chi_\alpha^{(n)} = \Phi_{2,n}(\alpha)$ for 
all $n\geq 2$ and $\alpha \in H^{>0}(GL_2 \F_q)$.
\end{remark}

The parabolic induction maps 
are compatible with one another under composition:
\begin{prop}
\label{prop:parindcompat}
For all $\ell\leq m\leq n$,
\[ \Phi_{m,n}\circ\Phi_{\ell,m}=\Phi_{\ell,n}. \]
\end{prop}
\begin{proof}
Consider the diagram
\begin{equation*}
\vcenter{\xymatrix@!0@C=6.5em@R=12ex{
    &
    H^\ast\left(
        {\begin{bmatrix}
		GL_{\ell} & * 
		\\ 
		& GL_{n-\ell}
        \end{bmatrix}}
	\right)
	\ar@{<-}@/_2pc/[dddl]_(0.45){\pi^\ast}
	\ar[rr]^{\id}
    \ar[dr]^{i^\ast}
	&&
    H^\ast\left(
        {\begin{bmatrix}
		GL_{\ell} & * 
		\\ 
		& GL_{n-\ell}
        \end{bmatrix}}
	\right)
    \ar@/^2pc/[dddr]^(0.45){i_!}
    \\
    &&
	H^\ast\left(		
        {\begin{bmatrix}
		GL_\ell & * & * 
		\\
		& GL_{m-\ell} & * 
		\\
		& & GL_{n-m}
        \end{bmatrix}}
	\right)
	\ar[ur]^{i_!}
    \ar@{<-}[dl]_-{\pi^\ast}
    \ar[dr]^-{i_!}
    \\
    &
	H^\ast\left(
        {\begin{bmatrix}
        GL_\ell & * 
        \\
        & GL_{m-\ell}
        \end{bmatrix}}
	\right)
    \ar@{<-}[dl]_-{\pi^\ast}
    \ar[dr]^-{i_!}
    &&
	H^\ast\left(
        {\begin{bmatrix}
        GL_m & * 
        \\
        & GL_{n-m}
        \end{bmatrix}}
	\right)
    \ar@{<-}[dl]_-{\pi^\ast}
    \ar[dr]^-{i_!}
    \\
    H^\ast (GL_\ell)
    &&
    H^\ast (GL_m)
    &&
    H^\ast (GL_n)
}} 
\end{equation*}
where for brevity we have written $GL_a$ for $GL_a\F_q$ and 
where the various $i$'s and $\pi$'s stand for the 
evident inclusion and projection maps, respectively.
Observe that the diagram commutes: For the two cells with curved
arrows, commutativity is immediate; for the diamond in the 
middle, commutativity follows from the fact that the two 
projection maps $\pi$ involved have the same kernel;
and for the triangle on top, commutativity follows from the 
observation that the index of the subgroup
\[
	\begin{bmatrix}
	GL_\ell & * & * 
	\\
	& GL_{m-\ell} & * 
	\\
	& & GL_{n-m}
	\end{bmatrix}
	\leq
    \begin{bmatrix}
	GL_{\ell} & * 
	\\ 
	& GL_{n-\ell}
    \end{bmatrix}
\]
equals the number of points in the Grassmannian 
$\Gr_{m-\ell}(\F_q^{n-\ell})$, which is 
\[
	{n-\ell \choose m-\ell}_q 
	=  \prod_{i=1}^{m-\ell}\frac{q^{n-\ell-i+1}-1}{q^i-1}
	\equiv 1  \mod q.
\]
The claim now follows by observing that the composite
from $H^\ast(GL_\ell)$ to $H^\ast(GL_n)$ along the 
bottom of the diagram equals $\Phi_{m,n}\circ \Phi_{\ell,m}$, 
while the composite along the top of the diagram equals $\Phi_{\ell,n}$.
\end{proof}

Combining Proposition~\ref{prop:parindcompat} and
Remark~\ref{rk:chiparindconn}, 
we obtain the desired connection between the 
classes $\chi_\alpha^{(n)}$ for different values of $n$.
\begin{cor}
$\chi_\alpha^{(n)}=\Phi_{m,n}(\chi_\alpha^{(m)})$ 
for all $2\leq m \leq n$
and $\alpha \in H^\ast(GL_2 \F_q).$
\qed
\end{cor}
\noindent
This in turn implies
\begin{cor}
\label{cor:parindclasses}
If $\chi_\alpha^{(n)}$ is nonzero or non-nilpotent, so is $\chi_\alpha^{(m)}$ for every $2 \leq m \leq n$. 
\end{cor}
\begin{proof}
Only the claim regarding non-nilpotence requires comment. 
Since non-nilpotence can be verified using Steenrod powers,
the claim follows by observing that the parabolic induction maps 
commute with the Steenrod operations, because induced maps and 
transfer maps~\cite{Evens} do. 
\end{proof}

In view of Corollary~\ref{cor:parindclasses},
to establish the nonvanishing of the universal classes 
$\chi_\alpha^{(m)}$
for a large range of $m$'s, one should try to find 
as high-dimensional representations 
$\rho$ as possible with $\chi_\alpha(\rho)\neq 0$.
This is our aim in section~\ref{largerep}.

\section{Vanishing on decomposables}\label{decomposables}
Our goal in this section is to prove the following result.
\begin{thm}\label{twofixedlines}
If $\rho$ is a representation which has a 
trivial subrepresentation of dimension 2, then \[ \chi_\alpha(\rho) = 0 \]
for all $\alpha$.
\end{thm}

\begin{remark}\label{rk:uppertriangularization}
Representations over $\F_q$ (or any field of characteristic $p$) have an ``upper-triangularization principle'' with respect to mod-$p$ cohomology.
Indeed, if $P\leq G$ is a Sylow $p$-subgroup, then restriction from $G$ to $P$ in cohomology is injective.  But, the restriction of any representation $\rho$ of $G$ to $P$ is unipotent upper-triangular with respect to some basis (because the image of $P$ in $GL(\rho)\iso GL_N\F_q$ is subconjugate to the $p$-Sylow subgroup of $GL_N\F_q$, consisting of the unipotent upper-triangular matrices).  In particular, the action of $P$ on $\rho$ has a fixed line. 
\end{remark}
Hence, Theorem~\ref{twofixedlines} implies:
\begin{cor}\label{cor:vanishingondecomposables}
If $\rho$ and $\eta$ are nonzero representations, then
\[
	\chi_\alpha(\rho\oplus\eta)=0
\]
for all $\alpha$. \qed
\end{cor}

\begin{remark}
For a representation of a $p$-group $P$ over a field of characteristic $p$, the condition of having just one-dimensional fixed-point space is quite restrictive.  Indeed, such representations are precisely the subrepresentations of the regular representation of $P$.
\end{remark}

\begin{remark}
Taking $\eta$ to be the trivial 1-dimensional representation shows
that the universal classes 
$\chi_\alpha^{(n)}\in H^*(GL_n\F_q)$ vanish upon restriction to $H^*(GL_{n-1}\F_q)$.  
That is, they are annihilated by the stabilization maps.
\end{remark}

\begin{remark}
\label{rk:primitives}
Corollary~\ref{cor:vanishingondecomposables} implies that the universal classes $\chi_\alpha^{(n)}$ are primitives in the bialgebra
\[ \bigoplus_{n=0}^\infty H^*(GL_n\F_q) \]
where the coproduct is induced by the block-sum homomorphisms
\[ GL_a\F_q\times GL_b\F_q\to GL_{a+b}\F_q. \]
\end{remark}

We now turn to the proof of Theorem~\ref{twofixedlines}.  
Let $V\leq\rho^G$ be a subspace of the set of fixed-points in $\rho$.  Define 
\[ \Emb^{(V)}(\F_q^2,\rho)=\set{f\in\Emb(\F_q^2,\rho)}{\im(f)\cap\rho^G=V}, \]
the set of embeddings whose image contains precisely the fixed-points in $V$ 
(nonempty only for $\dim V\leq2$).  
This decomposition
\[ \Emb(\F_q^2,\rho)=\coprod_{V\leq\rho^G} \Emb^{(V)}(\F_q^2,\rho) \]
as $(G\times GL_2\F_q)$-sets yields a disjoint union decomposition of the space
$\Emb(\F_q^2,\rho)\sslash G\times GL_2\F_q$.
Consequently, Definition~\ref{def:chiclassdef2} splits up as a sum
\[ \chi_\alpha(\rho)=\sum_{V\leq\rho^G} ((\pi_V)_!\circ\tau_V^\ast )(\alpha), \]
where $\pi_V$ and $\tau_V$ are the evident projections
fitting into the diagram
\[\xymatrix@!0@C=8em@R=8ex{
	&
	\Emb^{(V)}(\F_q^2,\rho)\sslash G\times GL_2\F_q
	\ar[dl]^{\pi_V}_{!}
	\ar[dr]_{\tau_V}
	\\
	\pt\sslash G
	&&
	\pt\sslash GL_2\F_q
}\]
The following lemma says that, in calculating $\chi_\alpha(\rho)$, one needs 
to consider only those embeddings whose image contains all of the fixed points,
and also contains nonfixed points.
\begin{lemma}\label{allfixedpoints}
Let $\rho$ be a representation of $G$, and $V\leq\rho^G$ a subspace. Then 
\[ (\pi_V)_!\circ\tau_V^\ast=0 \]
in positive degrees
if either
\begin{enumerate}
\item $\dim V=2$, or
\item $V<\rho^G$ is a proper subspace.
\end{enumerate}
\end{lemma}
\begin{proof}
First suppose $\dim V=2$.  Then the $G$-action on 
$\Emb^{(V)}(\F_q^2,\rho)$ is trivial, 
so $\tau_V$ factors through 
\[ \Emb^{(V)}(\F_q^2,\rho)\sslash GL_2\F_q=
\textrm{Iso}(\F_q^2,V)\sslash GL_2\F_q, \]
which is contractible.

Now we turn to the case in which $V\leq\rho^G$ is proper. 
If $\dim(V) > 2$, we are done since in this case 
$\Emb^{(V)} (\F_q^2,\rho)$
is empty. In light of the previous case, we may therefore 
assume that $\dim V<2$.
Pick some fixed line $L\leq\rho^G$ which is not contained in $V$.
Let $X$ be the image of $\Emb^{(V)} (\F_q^2,\rho)$ under the map
\[ \hom(\F_q^2,\rho)\longto\hom(\F_q^2,\rho/L) \]
induced by the quotient homomorphism $r\colon \rho\to\rho/L$.

Observe that the fibre over $f\in X$ of the resulting 
surjection 
\[ 
	\Emb^{(V)} (\F_q^2,\rho)\longto X 
\]
has a free $\hom(\F_q^2/f\inv rV, L)$-action.
Consequently, the cardinalities of the fibres 
are divisible by $q^{2-\dim(V)}$, and hence by $p$.
The same is true after passing to homotopy orbits, so that in the commutative diagram
\[\xymatrix@!0@C=8em@R=8ex{
	&
	X \sslash G\times GL_2\F_q
	\ar@/_3pc/[ddl]_(0.7){\pi'_V}
	\ar@/^3pc/[ddr]^(0.7){\tau'_V}
	\\
	&
	\Emb^{(V)} (\F_q^2,\rho) \sslash G\times GL_2\F_q
	\ar[u]_\phi
	\ar[dl]^{\pi_V}
	\ar[dr]_{\tau_V}
	\\
	\pt\sslash G
	&&
	\pt\sslash GL_2\F_q,
}\]
(where $\pi'_V$ and $\tau'_V$ are the evident projections),
the map $\phi$ is a union of covering spaces
whose multiplicities are divisible by $p$.
It follows that
\[ (\pi_V)_!\circ\tau_V^\ast 
= (\pi'_V)_!\circ(\phi_!\circ \phi^\ast)\circ (\tau'_V)^\ast 
= (\pi'_V)_!\circ 0 \circ (\tau'_V)^\ast
= 0.\qedhere \]
\end{proof}

Given a subspace $V \leq \rho^G$, let us
write $\Emb^{(\geq V)}(\F_q^2,\rho)$ for the set of 
embeddings whose image contains $V$. For later reference, we 
record the following corollary of Lemma~\ref{allfixedpoints}.
\begin{cor}
\label{cor:embatleastv}
Let $\rho$ be a representation of $G$ and let $V \leq \rho^G$
be a subspace.
For all $\alpha \in H^{>0}(GL_2\F_q)$ we have 
\[
	\chi_\alpha(\rho) = (\pi_! \circ \tau^\ast)(\alpha)
\]
where $\pi$ and $\tau$ are the evident projections fitting into
the diagram
\[
\pushQED{\qed} 
\begin{gathered}[b]
\xymatrix@!0@C=8em@R=8ex{
	&
	\Emb^{(\geq V)}(\F_q^2,\rho)\sslash G\times GL_2\F_q
	\ar[dl]^{\pi}_{!}
	\ar[dr]_{\tau}
	\\
	\pt\sslash G
	&&
	\pt\sslash GL_2\F_q 
}
\\[-\dp\strutbox]
\end{gathered}
\qedhere
\popQED
\]
\end{cor}

\begin{proof}[Proof of Theorem~\ref{twofixedlines}]
Since we are assuming $\dim\rho^G\geq2$, any $V\leq\rho^G$ satisfying $\dim V<2$ is proper.  So by the lemma, $(\pi_V)_!\circ\tau_V^\ast=0$ for every $V$, 
showing $\chi_\alpha(\rho)=0$.
\end{proof}

\section{The cohomology of $GL_2\F_q$}\label{cohofGL2}
Because our family of characteristic classes is indexed on the modular cohomology of $GL_2\F_q$, 
we now describe the cohomology for the reader's convenience.
The result described in this section is 
well-known 
\cite{Aguade,QuillenK,FriedlanderParshall}~\cite[ch.~1]{Ddissertation}.
To make the description more explicit, we assume in this section
that the coefficient field $\F$ for cohomology is an extension of $\F_q$.

The $p$-Sylow subgroup of $GL_2\F_q$ is
\[ \F_q=\begin{bmatrix}1&\ast\\&1\end{bmatrix}, \]
the additive group of the field $\F_q$, which is elementary abelian.  Hence~\cite[Thm.~II.6.8]{A-M} restriction gives an isomorphism
\[ H^*(GL_2\F_q)\iso H^*(\F_q)^{\F_q^\times} \]
with the invariants of the multiplicative group 
\[ \F_q^\times\iso C_{q-1} \] 
acting on the cohomology ring of the additive group.

\begin{remark}\label{aff1}
The same remarks apply with $GL_2\F_q$ replaced by 
its subgroup
\[ \Aff_1\F_q=\begin{bmatrix}1&\ast\\0&\ast\end{bmatrix}. \] 
Hence
the restriction in mod $p$ cohomology from $GL_2\F_q$ to $\Aff_1\F_q$
is an isomorphism.  Since the index is a unit modulo $p$, the transfer
$H^*(\Aff_1\F_q)\to H^*(GL_2\F_q)$ is an isomorphism as well.
\end{remark}

By the assumption that the coefficient field $\F$ is an extension 
of $\F_q$, the action of $\F_q^\times$ on $H^\ast(\F_q)$
diagonalizes, as all $(q-1)$-th roots are present and distinct. 
That is, there is a weight-space decomposition
\begin{equation}
\label{eq:weightspacedecomp}
H^*(\F_q) = \bigoplus_{k\in\bbZ/(q-1)} E_k, 
\end{equation}
with $E_k$ consisting of those elements on which $c\in\F_q^\times$ acts as multiplication by $c^k$.
In particular, restriction gives an isomorphism
\[ H^*(GL_2\F_q)\iso H^*(\F_q)^{\F_q^\times}=E_0. \]

Now, one can find canonical (up to scalar multiple) generators for the cohomology:
\[ H^*(\F_{p^r}) = \begin{cases}
\F[y_0,\dots, y_{r-1}] &\mbox{if $p=2$},\\
\F\langle x_0,\dots,x_{r-1}\rangle \\
\qquad\otimes
\F[y_0,\dots,y_{r-1}] &\mbox{if $p$ odd},
\end{cases} \]
where \[ x_k,y_k\in E_{p^k} \]
and $\deg x_i =1$, $\deg y_i =2$ (for $p=2$, $\deg y_i = 1$).
The angle braces indicate an exterior algebra.
The monomials
\[ x^Ay^B = \prod_{k=0}^{r-1}x_k^{a_k}y_k^{b_k} \]
for \[ A=(a_0,\dots,a_{r-1})\in\{0,1\}^r \+ B=(b_0,\dots,b_{r-1})\in\bbN^r \]
now form an additive basis for the cohomology, consisting of eigenvectors for the action of the multiplicative group.

From this, we obtain the following description of the invariants:
\begin{lemma}\label{GL2}
Under the assumption that $\F$ is an extension of $\F_q$,
the restriction \[ H^*(GL_2\F_q;\,\F)\longto H^*(\F_q;\,\F) \]
is an isomorphism onto the subspace (and subring) generated by the monomials
\[ x^Ay^B \]
with the property that
\[
	(q-1)\mid\sum_{k=0}^{r-1}p^k\left(a_k+b_k\right),
\]
where $q=p^r$.
(In the case of odd $p$; for $p=2$, of course there are no exterior terms.)
\qed
\end{lemma}

\section{The coproduct on the cohomology of $GL_2\F_q$}\label{coprod}
Because $\F_q$ is an abelian group, its addition map 
$\mu \colon \F_q \times \F_q \to \F_q$ gives its cohomology 
the structure of a (coassociative cocommutative counital) coalgebra. 
We will show that, although $GL_2\F_q$ is not abelian, its cohomology 
inherits a coproduct from the Sylow $p$-subgroup $\F_q$ of $GL_2\F_q$.
Observe that the transfer map 
\[
	\trna \colon H^\ast(\F_q) \longto  H^\ast(GL_2\F_q) 
\]
is a retraction onto the image of $\textrm{res}:H^\ast(GL_2\F_q) \incl H^\ast(\F_q)$.
\begin{prop}
\label{prop:coproduct}
For any field $\F$ of characteristic $p$, 
the group cohomology 
$H^\ast(GL_2\F_q;\,\F)$ admits a unique 
coassociative cocommutative counital coproduct $\Delta$
making 
\[
	\trna \colon H^\ast(\F_q;\,\F) \longto  H^\ast(GL_2\F_q;\,\F) 
\]
into a homomorphism of coalgebras.
\end{prop}
\begin{proof}
It suffices to verify that the kernel $J = \ker(\trna)$
is a coideal, that is, that $J$ is annihilated by the 
counit $H^\ast(\F_q;\,\F) \to \F$ and that 
\[
	\mu^\ast J
	\subset 
	H^\ast(\F_q;\,\F) \tensor_\F J + J\tensor_\F H^\ast(\F_q;\,\F).
\]
The first of these conditions is immediate form the observation 
that $\trna$ is an isomorphism in degree 0. 
To verify the 
second condition, observe that if it holds for the field $\F$,
it also holds for any subfield of $\F$, 
by the compatibility of 
transfers and induced maps with extending the coefficient field
for cohomology. Thus it is enough to check the condition 
in the case where $\F$ is an extension of $\F_q$. 
In that case, we have the weight-space decomposition
\eqref{eq:weightspacedecomp}, 
with respect to which the inclusion 
$H^\ast(GL_2\F_q) \incl H^\ast(\F_q)$
corresponds to the inclusion 
$E_0 \incl \bigoplus_k E_k$ and the 
transfer $\trna$ to the projection 
$\bigoplus_k E_k \to E_0$. In particular, we have 
\[
	J= \bigoplus_{k\neq 0} E_k.
\]
Because the addition map $\mu:\F_q\times\F_q\to\F_q$ 
is equivariant with respect to the action of the 
multiplicative group, the induced map 
$\mu^*\colon H^*(\F_q)\to H^*(\F_q)\otimes H^*(\F_q)$ satisfies
\[ \mu^*(E_k) \subset \bigoplus_{i+j=k}E_i\otimes E_j. \]
Thus
\begin{align*}
\mu^*(J) 
&\subset\bigoplus_{i+j\neq0}E_i\otimes E_j\\
&\subset\bigoplus_{\substack{i\neq0\textrm{ or }\\j\neq0}}E_i\otimes E_j\\
&=J\otimes H^*(\F_q;\,\F) + H^*(\F_q;\,\F)\otimes J.
\qedhere
\end{align*}
\end{proof}

We warn the reader that the coproduct does not make 
$H^*(GL_2\F_q)$ into a Hopf algebra, 
as it is not compatible
with the cup product.

Explicitly, we can describe $\Delta$ 
in terms of the  monomial basis of Lemma~\ref{GL2}
under the assumption that $\F$ is an extension of $\F_q$:
\[ \Delta\left(x^Ay^B\right)=\sum_{\substack{A'+A''=A\\B'+B''=B\\P(A'+B')}}\binom{B}{B'}
x^{A'}y^{B'}\otimes x^{A''}y^{B''}. \]
Here we have used the shorthand
\[ \binom{B}{B'} = \binom{b_0}{b_0'}\cdots\binom{b_{r-1}}{b_{r-1}'}, \]
and the sum runs over only those pairs with the divisibility property
\begin{equation}
\label{eq:divprop}
P(C) = \left[ (p^r-1)\mid\sum_{k=0}^{r-1}p^kc_k \right]. 
\end{equation}
(As usual, for $p=2$ one must remove the exterior terms.)
Informally, one simply performs the usual coproduct on $H^*(\F_q)$ 
(which is multiplicative, 
with the generators primitive), and then throws out all 
terms which do not satisfy the divisibility condition~\eqref{eq:divprop}.

\section{Wedge sum formula}\label{wedge}
While our characteristic classes vanish on direct sums by 
Corollary~\ref{cor:vanishingondecomposables}, they exhibit
interesting behavior with respect to the following wedge sum 
construction, which we can use to combine two representations 
without proliferating their fixed-points.

\begin{defn}
\label{def:pointedrep-wedgesum}
A \emph{pointed representation} of a group $G$ is 
a representation $\rho$ of $G$ 
equipped with the choice of a
\emph{basepoint}, a non-zero vector $v_0 \in \rho$ fixed
by the $G$-action.
If $(\rho,v_0)$ and $(\eta,w_0)$ are  pointed representations 
of groups $G$ and $H$, respectively, 
we define their \emph{wedge sum} to be the representation 
\[ \rho\vee\eta = \rho\oplus\eta/\langle v_0-w_0 \rangle, \]
of $G\times H$,
a pointed representation with basepoint $v_0=w_0$.
The isomorphism type of 
$\rho\vee\eta$
as a pointed representation remains unchanged if the
basepoints $v_0$ and $w_0$ are replaced by non-zero multiples.
Hence for representations with unique fixed lines,
we will take the liberty to form wedge sums without explicitly specifying the basepoints.
\end{defn}
\begin{remark}
Working with pointed representations 
may appear overly restrictive, since not all representations admit 
a non-zero fixed vector. However, from the point of view of 
mod-$p$ cohomology it is no loss, due to the 
upper-triangularization principle of 
Remark~\ref{rk:uppertriangularization}.
\end{remark}

We describe how to calculate the classes of a wedge sum in terms
of the classes of the two factors:
\begin{thm}\label{wedgeformula}
If $\rho$ and $\eta$ are pointed representations of $G$ and $H$, respectively, then for all $\alpha\in H^*(GL_2\F_q)$
\begin{align} \label{eq:wedgeformula}
	\chi_\alpha(\rho\vee\eta)
	&=
	-\sum \chi_{\alpha_{(1)}}(\rho)\times\chi_{\alpha_{(2)}}(\eta)
\end{align}
as an element of $H^*(G\times H)$,
where
\[ \Delta(\alpha)=\sum \alpha_{(1)}\otimes\alpha_{(2)} \]
is the coproduct constructed in section~\ref{coprod}.
\end{thm}

We begin the proof by establishing an ``affine version'' of the push-pull construction in
Definition~\ref{def:chiclassdef2},
better suited to working with pointed representations.
Let $i:\F_q\to GL_2\F_q$ be the inclusion.
\begin{prop}
\label{prop:pointedrepformula}
Let $(\rho,v_0)$ be a pointed representation of a group $G$
with $\dim(\rho) \geq 2$.
The characteristic classes of $\rho$ are given by following the diagram
\[\xymatrix{
	& 
	\rho \sslash G\times\F_q
	\ar[dl]_-{!}^{\pi}
	\ar[dr]_-{\tau}
	\\
	\pt\sslash G
	&&
	\pt \sslash \F_q
	\ar[r]^-{i}
	&
	\pt\sslash GL_2\F_q
}\]
where the action of $G\times \F_q$ on $\rho$ is given by
\[ (g,c)\cdot v = gv - cv_0. \]
More precisely,
\[
	\chi_\alpha(\rho) = -(\pi_!\circ\tau^\ast\circ i^*)(\alpha)
\]
for all $\alpha \in H^*(GL_2\F_q)$.
\end{prop}
\begin{proof}
We first observe that the formula is correct when $\deg\alpha=0$.  In that case, 
$\chi_\alpha(\rho)$
is zero by convention, while the right hand side is multiplication by the
multiplicity of the covering $\pi$, which is 
\[ |\rho/\langle v_0\rangle|=q^{\dim\rho-1}, \]
a multiple of $q$ (so zero on cohomology).

Hence we assume $\alpha\in H^{>0}(GL_2\F_q)$.
Now, the expression on the right hand side splits as a sum of two terms, corresponding to the
equivariant decomposition
\[ \rho = \langle v_0\rangle \coprod (\rho-\langle v_0\rangle). \]
The former term vanishes, because
$\langle v_0\rangle\sslash G\times\F_q\to\pt\sslash\F_q$ factors through
the contractible space $\langle v_0\rangle\sslash\F_q$.
So we must verify that the latter term in the sum agrees with $\chi_\alpha(\rho)$.
Consider the commutative diagram
\[ \xymatrix@C+0.5em{
	\pt\sslash G
	&
	\Emb^{(\geq\langle v_0\rangle)}(\F_q^2,\rho)\sslash G\times GL_2\F_q
	\ar[l]^-{!}_-{\pi'}
	\ar[r]^-{\tau'}
	&
	\pt\sslash GL_2
	\\
	&
	\rho-\langle v_0\rangle\sslash G\times\F_q
	\ar[ul]^-{!}_-{\pi''}
	\ar[u]^-{\phi}
	\ar[r]^-{\tau''}
	&
	\pt\sslash\F_q
	\ar[u]^{i}
} \]
where $\Emb^{(\geq\langle v_0\rangle)}(\F_q^2,\rho)$
is as in Corollary~\ref{cor:embatleastv}
and where $\phi$ is induced by the map sending 
$v\in\rho-\langle v_0\rangle$ to the embedding $e_1\mapsto v_0$, $e_2\mapsto v$.

Following the diagram along the bottom, 
i.e.~$\pi''_!\circ(\tau'')^*\circ i^*$,
yields the right hand side of the proposition, without the minus sign.
Meanwhile, following the top of the diagram,
i.e.~$\pi'_!\circ(\tau')^*$,
yields 
$\alpha\mapsto\chi_\alpha(\rho)$, 
by Corollary~\ref{cor:embatleastv}.
Now, up to homotopy equivalence, $\phi$ agrees with the covering map
\[ \left(\frac{\rho}{\langle v_0\rangle}-0\right)\sslash G
	\longto \bbP\left(\frac{\rho}{\langle v_0\rangle}\right)\sslash G  \]
which has multiplicity $q-1$.
(We used that the actions of $GL_2\F_q$ and $\F_q$ are free, to replace their
homotopy quotients with strict quotients.)
So
$\phi_!\circ\phi^*=-1$, and we get
\begin{align*}
	\chi_\alpha(\rho) 
	&= (\pi'_!\circ\tau'^*)(\alpha)\\ 
	&= - (\pi'_!\circ(\phi_!\circ\phi^*)\circ{\tau'}^*)(\alpha) \\
	&= - (\pi''_!\circ{\tau''}^*\circ i^*)(\alpha).\qedhere
\end{align*}
\end{proof}

We will also need the following fact:
\begin{lemma}\label{hatclassesagree}
In the notation of Proposition~\ref{prop:pointedrepformula},
\[ \pi_!\circ\tau^*\circ (i^*\circ i_!) = \pi_!\circ\tau^*. \]
\end{lemma}
\begin{proof}
Because $i_!\circ i^*=1$, it suffices to show that 
$\pi_!\circ\tau^*\colon H^*(\F_q)\to H^*(G)$
factors through $i_!\colon H^*(\F_q)\to H^*(GL_2\F_q)$.
To do so, it suffices to show that it factors through the transfer
$H^*(\F_q)\to H^*(\Aff_1\F_q)$,
because the further transfer 
$H^*(\Aff_1\F_q)\to H^*(GL_2\F_q)$ is an isomorphism 
by Remark~\ref{aff1}.
We have a commutative diagram
\[ \xymatrix{
	\pt\sslash G
	&
	\rho\sslash G\times \Aff_1\F_q
	\ar[l]_-{!}
	\ar[r]^-{\tau'}
	&
	\pt\sslash \Aff_1\F_q
	\\
	&
	\rho\sslash G\times\F_q
	\ar[ul]^-{\pi}_{!}
	\ar[u]^{\phi}_{!}
	\ar[r]^-{\tau}
	&
	\pt\sslash\F_q
	\ar[u]^{j}_{!}
} \]
where
$\begin{bmatrix}1&a\\0&b\end{bmatrix}\in\Aff_1\F_q$
acts on $\rho$ by $v\mapsto b^{-1}v-b^{-1}av_0$.
Here $\phi$ and $j$ are homotopy 
equivalent to finite covering maps via equivalences under which
the map $(\tau,\tau')$ corresponds to a map of
covering spaces inducing a bijection on the fibers.  Hence $(\tau,\tau')$ intertwines the two transfer maps, i.e.
\[ \phi_!\circ\tau^*={\tau'}^*\circ j_!. \]
This yields the desired factorization.
\end{proof}

\begin{proof}[Proof of Theorem~\ref{wedgeformula}]
Let $\mu\colon \F_q \times \F_q \to \F_q$ be the addition map,
and $i\colon\F_q\to GL_2\F_q$ the inclusion.
Consider the commutative diagram
\[ \xymatrix{
	\pt \sslash G\times H
	&
	\rho\vee\eta \sslash G\times H \times \F_q
	\ar[l]_-{!}^-{\pi'}
	\ar[r]_-{\tau'}
	&	
	\pt \sslash \F_q
	\ar[r]_-{i}
	&
	\pt\sslash GL_2\F_q
	\\
	&
	\ar[u]_-{\homot}^-\varphi
	\rho\times\eta\sslash G\times \F_q \times H\times \F_q
	\ar[ul]_-{!}^-{\pi}
	\ar[r]_-{\tau}
	&
	\pt \sslash \F_q\times\F_q
	\ar[u]^-{\mu}
} \]
Here $\varphi$ is induced by the quotient map 
$\rho\times\eta \to \rho \vee \eta$ and the addition map $\mu$.
The action of $(G\times \F_q) \times (H\times \F_q)$
on $\rho\times\eta$ is the product of those in
Proposition~\ref{prop:pointedrepformula}.
Notice that $\pi$ and $\tau$ are each a product of two projection maps.
Observe also that the kernel of $\mu$ 
acts freely on $\rho \times \eta$,
with quotient $\rho\vee\eta$. Thus the map $\varphi$ is a homotopy 
equivalence.

Applying Proposition~\ref{prop:pointedrepformula} 
to both $\rho$ and $\eta$ 
shows that the right hand side of 
equation~\eqref{eq:wedgeformula} agrees with
$-(\pi_!\circ\tau^\ast\circ (i\times i)^*)(\Delta\alpha)$.
Inserting 
\[ \Delta=(i\times i)_!\circ\mu^*\circ i^* \]
and applying Lemma~\ref{hatclassesagree}
to remove $(i\times i)^*(i\times i)_!$ yields
$-(\pi_!\circ\tau^\ast\circ\mu^*\circ i^*)(\alpha). $ 
The diagram  shows that this is equal to 
$-(\pi'_!\circ{\tau'}^\ast\circ i^*)(\alpha)$,
which by Proposition~\ref{prop:pointedrepformula}
agrees with the left hand side of equation~\eqref{eq:wedgeformula}.
\end{proof}

\section{The basic representations}\label{basicrep}
We now define a family of representations whose characteristic classes we shall be able to describe.
For any $\F_q$-vector space $V$, regard $V$ as an elementary abelian group under addition, and define its ``basic representation'' $\rho_V$ to be
\[ \rho_V=\F_q\oplus V^* \]
as a vector space, with $V$ acting as
\[ v\cdot(c,u)=(c+u(v),u). \]
In coordinates, it looks like
\[ \rho_{\F_q^n}:(a_1,\dots,a_n)\longmapsto\begin{bmatrix}1&a_1&\cdots&a_n\\&1\\&&\ddots\\&&&1\end{bmatrix}. \]
We equip $\rho_V$ with the basepoint $(1,0) \in \F_q \oplus V^*$.
If $V$ and $W$ are two vector spaces, we then have
\[ \rho_V\vee\rho_W=\rho_{V\oplus W}. \]
Therefore, in terms of a basis we have a decomposition
\[ \rho_{\F_q^n}=\rho_{\F_q}\vee\cdots\vee\rho_{\F_q} \]
as representations of $\F_q^n$.  Hence by Theorem~\ref{wedgeformula},
\[ \chi_\alpha(\rho_{\F_q^n})= (-1)^{n-1} \sum \chi_{\alpha_{(1)}}(\rho_{\F_q})\times\cdots\times\chi_{\alpha_{(n)}}(\rho_{\F_q}). \]
But $\rho_{\F_q}$ is the (2-dimensional) representation given by the inclusion of 
$\F_q$ as the Sylow $p$-subgroup of $GL_2\F_q$, by which we identified 
$H^*(GL_2\F_q)$ as a subspace of $H^*(\F_q)$.  Therefore by 
Theorem~\ref{thm:chiclassprops}.(\ref{thmpart:props-nontriviality}),
\[ \chi_\alpha(\rho_{\F_q^n})=\begin{cases}\alpha&\textrm{if }\deg\alpha>0,\\0&\textrm{if }\deg\alpha=0.\end{cases} \]
Consequently, $\chi_\alpha(\rho_{\F_q^n})$
is obtained from the iterated coproduct $\Delta^{n-1}(\alpha)$ by throwing away all those terms which have degree 0 in some factor.
Let us write the result explicitly, assuming that $\F$ extends $\F_q$.
To begin with, write the iterated coproduct as
\[ \Delta^{n-1}\left(x^Ay^B\right)=\sum_{\substack{A_1+\cdots+A_n=A\\B_1+\cdots+B_n=B\\P(A_i+B_i)\ \forall i}}\binom{B}{B_1;\cdots;B_n}
x^{A_1}y^{B_1}\otimes\cdots\otimes x^{A_n}y^{B_n}. \]
Here we have used the shorthand
\begin{equation}
\label{eq:multinomialshorthand}
\binom{B}{B_1;\cdots;B_n} = \binom{b_0}{b_{0,1},\dots,b_{0,n}}\cdots\binom{b_{r-1}}{b_{r-1,1},\ldots,b_{r-1,n}}, 
\end{equation}
where the terms on the right hand side are multinomial coefficients, and 
$B_i=(b_{0,i},\dots,b_{r-1,i})$.
Now,
\begin{align}\label{multinomialformula}
\chi_{x^Ay^B}(\rho_{\F_q^n})=(-1)^{n-1}\!\!\!\!\!\!\!\!\sum_{\substack{A_1+\cdots+A_n=A\\B_1+\cdots+B_n=B\\P(A_i+B_i)\ \forall i\\A_i+B_i\neq0\ \forall i}}\binom{B}{B_1;\cdots;B_n}
x^{A_1}y^{B_1}\times\cdots\times x^{A_n}y^{B_n}.
\end{align}

\section{Nontrivial classes}\label{nontriviality}
We now consider the question of which characteristic classes of the basic representation are nonzero.  In equation~\eqref{multinomialformula}, the monomials $x^{A_1}y^{B_1}\times\cdots\times x^{A_n}y^{B_n}$ in the sum are all linearly independent.  So
\[ \chi_{x^Ay^B}(\rho_{\F_q^n})\neq0 \]
if and only if there exists a nonzero term in the sum.
By iterating Lucas' theorem on the value of binomial coefficients mod $p$
(sse e.g.\ \cite{Fine}),
we see that the coefficient \eqref{eq:multinomialshorthand}
appearing in \eqref{multinomialformula} is nonzero mod $p$
precisely when
there is not a carry when adding together the numbers
$b_{j,1},\ldots, b_{j,n}$ in base $p$
for any $0 \leq j \leq r-1$.
It is now straightforward to check that by choosing
\[ A_i=(0,\dots,0) \+ B_i=p^{i-1}(p-1) \cdot (1,\dots,1), \]
we obtain a nonzero term appearing in a sum of type
\eqref{multinomialformula}, as we do by choosing
\[ A_1=(1,\dots,1) \+ B_1=(p-2)\cdot (1,\dots,1), \]
and for $i>1$ 
\[ A_i=(0,\dots,0) \+ B_i=p^{i-2}((p-2)p+1)\cdot (1,\dots,1). \]
Summing over $i$ from $1$ to $n$ to obtain $A$ and $B$, we have
the following result.
\begin{prop}\label{nonvanishingforbasicreps}
For p odd:
\begin{enumerate}[(i)]
\item \label{proppart:non-nilpotent}
	The class
	\[ \chi_\alpha(\rho_{\F_q^n})\in H^{2r(p^n-1)}(\F_q^n;\,\F_q) \]
    is non-nilpotent for 
    \[ 
    	\alpha =  (y_0\cdots y_{r-1})^{p^n-1}.
    \]
\item The class
	\[ \chi_\alpha(\rho_{\F_q^n})\in H^{r(2p^n-2p^{n-1}-1)}(\F_q^n;\,\F_q) \]
	is nonzero for
    \[
    	\alpha = x_0\cdots x_{r-1}(y_0\cdots y_{r-1})^{p^n-p^{n-1}-1}.
    \]
\end{enumerate}
Part (\ref{proppart:non-nilpotent})
is also valid for $p=2$, with the degree halved. \qed
\end{prop}

Now restrict to the case $r=1$, where we can give a complete description of which classes are nonzero.
\begin{defn}
For $m\in\bbN$, define $s_p(m)$ to be the sum of the digits of $m$ 
in $p$-ary notation.  That is,
\[ s_p(m)=\sum_k c_k \]
where
\[ m=\sum_k c_kp^k \quad\textrm{with}\quad 0\leq c_k<p\quad \forall k. \]
\end{defn}

First we give a bound for $m$ in terms of $s_p(m)$.
\begin{lemma}\label{sbound}
Let $s\in\bbN$. The lowest $m$ such that $s_p(m)=s$ is given by
\[ m=(d+1)p^c-1, \]
where $c$ and $d$ are determined by
\[ s=c(p-1)+d, \qquad 0\leq d<p-1. \]
\end{lemma}
\begin{proof}
The lowest $m$ must have a $p$-ary representation of the form
\[ d'.\underbrace{(p-1).\cdots.(p-1)}_{c'} \quad\textrm{with}\quad 0\leq d'<p-1  \]
since, if not, we may decrease $m$ while preserving $s_p(m)$ by ``moving digits to the right.''  For the digits to sum to $s$, we must have $d'=d$ and $c'=c$, which yields the claimed value for $m$.
\end{proof}

By Lemma~\ref{GL2} and the congruence
\[
	m \equiv s_p(m) \mod {p-1},
\]
the characteristic class $\chi_{y^m}$
is defined if and only if $s_p(m)$ is a multiple of $p-1$.
Similarly, the characteristic class
$\chi_{xy^m}$ is defined precisely when 
\[
	s_p(m) \equiv -1  \mod {p-1}.
\]
The following proposition determines in terms of $s_p(m)$
which of the classes $\chi_{y^m}(\rho_{\F_p^n})$ and 
$\chi_{xy^m}(\rho_{\F_p^{n}})$ are nonzero.

\begin{prop}\label{primecasenonvanishing}
Fix $n\geq1$.
\begin{enumerate}[(i)]
\item Let $p=2$.
For all $m\geq1$,
the class
\[ \chi_{y^m}(\rho_{\F_2^n})\in H^m(\F_2^n;\,\F_2) \]
is nonzero (and non-nilpotent) if and only if
\[ s_2(m)\geq n. \]
The lowest-dimensional such class occurs in degree $2^n-1$.
\item \label{proppart:r1non-nilpotent}
Let $p$ be odd.
The class
\[ \chi_{y^m}(\rho_{\F_p^n})\in H^{2m}(\F_p^n;\,\F_p) \]
is non-nilpotent if
\[ s_p(m)=k(p-1)\quad\textrm{with}\quad k\geq n. \]
The lowest-dimensional such class occurs in degree $2p^n-2$.
All other $\chi_{y^m}(\rho_{\F_p^n})$ are zero.
\item \label{proppart:r1nonzero}
Let $p$ be odd.
The class
\[ \chi_{xy^m}(\rho_{\F_p^n})\in H^{2m+1}(\F_p^n;\,\F_p) \]
is nonzero if 
\[ s_p(m)=k(p-1)-1\quad\textrm{with}\quad k\geq n. \]
The lowest-dimensional such class occurs in degree
$2p^n-2p^{n-1}-1$.
All other $\chi_{xy^m}(\rho_{\F_p^n})$ are zero.
\end{enumerate}
\end{prop}

\begin{proof}
We will prove part (\ref{proppart:r1nonzero}); 
the other parts are quite similar except that $x$ does not appear.

By writing $m$ in $p$-ary notation, we can form
a multiset
\[ \{x,y,\dots,y,y^p,\dots,y^p,\dots,y^{p^k},\dots,y^{p^k}\}, \]
where each power of $p$ appears less than $p$ times, and the product of the elements 
is our class $xy^m$.  
The number of elements is $1+s_p(m)$. 

Now, assuming $1+s_p(m)=k(p-1)$ with $k\geq n$,
the multiset can be partitioned 
into a disjoint union of $n$ nonempty submultisets, 
each of which has a multiple of $p-1$ elements.  
Choosing such a partition yields a factorization
$xy^m=(x^{a_1}y^{b_1})\cdots(x^{a_n}y^{b_n})$
such that $a_i+b_i \equiv s_p(a_i+b_i) \equiv 0$ mod $p-1$
for all $i$  and
there is no carry when  the numbers
$b_1,\ldots,b_n$ are added together in base $p$.
In view of the analysis of the multinomial coefficients
at the beginning of the 
section, this shows the existence of a nonzero summand on the right hand side in equation~\eqref{multinomialformula}.

Conversely, suppose $\chi_{xy^m}(\rho_{\F_p^n})\neq0$.
Let $A_i$, $B_i$, $i=1,\ldots,n$, be the exponents appearing in
some nonzero summand in equation~\eqref{multinomialformula}
(with $A=1$, $B=m$).
The condition
\[ \binom{m}{B_1,\dots,B_n}\neq0\ \mod p \]
implies that
\[ s_p(m)=s_p(B_1)+\cdots+s_p(B_n). \]
But each $A_i+B_i>0$, and
\[ A_i+s_p(B_i) \equiv A_i+B_i\equiv  0 \mod (p-1) \]
by the condition $P(A_i+B_i)$.
Thus $A_i+s_p(B_i)$ is a positive multiple of $p-1$ for all $i$. 
Summing up,
\[ 1+s_p(m) = k(p-1) \]
for some $k \geq n$, as desired.

The statement about degrees follows by applying Lemma~\ref{sbound}.
\end{proof}

\section{Representations of elementary abelian groups}\label{elemab}

Let $G$ be a $p$-group, and 
let $M$ be a representation of $G$; 
we remind the reader of our convention that this means
a finite dimensional representation over the field $\F_q$
of characteristic $p$.
We define a filtration
\[ 
	0 = J_{-1}(M) \leq J_0(M) \leq J_1(M) \leq \cdots \leq M
\]
of $M$ by setting $J_{-1}(M) = 0$ and inductively defining 
$J_\ell (M)$ to be the preimage of the fixed-point subspace
 $(M/J_{\ell-1}(M))^G$ under the quotient map $M \to M/J_{\ell-1}(M)$.
In particular, $J_0(M) = M^G$.
More generally,
\[ J_i(M)=\set{v\in M}{I^{i+1}\cdot v=0} \]
where $I \subset \F_qG$ is the augmentation ideal of the 
group ring.
By Remark~\ref{rk:uppertriangularization}, every inclusion in the 
filtration is strict until $J_\ell(M)$ becomes equal to $M$, 
so $J_{\dim(M)-1}(M) = M$.
As the filtration is preserved by maps between representations, 
it follows in particular that every $d$-dimensional submodule of $M$ must
be contained in $J_{d-1}(M)$.

Now we consider the case where $G$ is elementary abelian.
\begin{thm}\label{J1}
Suppose $\xi$ is a representation of an elementary abelian $p$-group.  Then
\[ \chi_\alpha(\xi)=\chi_\alpha(J_1(\xi)). \]
\end{thm}
\noindent That is, only the subrepresentation $J_1(\xi)$ matters when calculating $\chi_\alpha(\xi)$.
\begin{proof}
Definition~\ref{def:chiclassdef2} involves the transfer associated to the covering
\[ \Emb(\F_q^2,\rho)\sslash G\times GL_2\F_q\homot\Gr_2(\rho)\sslash G
\to\pt\sslash G. \]
This decomposes as a sum of the transfer maps associated to each of the components
of the cover.  Since $G=\pi_1(\pt\sslash G)$ is elementary abelian, all such 
transfers vanish except for those associated to trivial (1-sheeted) components.  These correspond to the 2-planes 
$V\in\Gr_2(\rho)$
such that $g\cdot V=V$ for all $g$, i.e., the 2-dimensional subrepresentations of $\rho$.  Since all
such subrepresentations are contained in $J_1(\rho)$, the result is unchanged by replacing the
cover $\Gr_2(\rho)$ with its subspace $\Gr_2(J_1(\rho))$.
\end{proof}

The $J_\bullet$ filtration is compatible with external tensor products
of representations:
\begin{lemma}
\label{lm:jandtensor}
Let $\xi_1$ and $\xi_2$ be representations of $p$-groups $G_1$ and $G_2$,
respectively. Then
\[
	J_i(\xi_1\tensor \xi_2) = \sum_{a+b=i} J_a(\xi_1) \tensor J_b(\xi_2) \subset \xi_1\tensor\xi_2
\]
where the tensor products are external, i.e.,\ representations of $G_1\times G_2$.
\end{lemma}
\begin{proof}
The claim clearly holds
for $i=-1$. Assume inductively that it holds for $i-1$. For $i\geq 0$,
the submodule $J_i(\xi_1\tensor \xi_2)$ consists precisely of those
vectors $v\in \xi_1\tensor \xi_2$ for which
$I\cdot v \subset J_{i-1}(\xi_1\tensor \xi_2)$,
where 
\[
	I 
	= 
	I_1 \tensor \F_q[G_2] + \F_q[G_1] \tensor I_2 
	\subset 
	\F_q[G_1\times G_2]
\]
is the augmentation ideal. Here $I_1$ and $I_2$ denote the 
augmentation ideals of $\F_q[G_1]$ and $\F_q[G_2]$, respectively.
Using this observation and the inductive assumption,
it is easily verified that 
$\sum_{a+b=i} J_a(\xi_1) \tensor J_b(\xi_2) \subset J_i(\xi_1\tensor \xi_2)$. 
To show the reverse containment, suppose 
$v \in J_i(\xi_1\tensor \xi_2)$. 
For $j=1,2$, working inductively on $k$,
we may find for each $k\geq 0$ an integer $d_k \geq 0$ and 
vectors  $w^{(j)}_{k,1},\ldots, w^{(j)}_{k,d_k}\in J_k(\xi_j)$
such that the vectors
\[
	w^{(j)}_{0,1},\ldots,w^{(j)}_{0,d_0};\,
	w^{(j)}_{1,1},\ldots,w^{(j)}_{1,d_1};\,
	\ldots;\,
	w^{(j)}_{k,1},\ldots,w^{(j)}_{k,d_k}	
\]
form a basis of $J_k(\xi_j)$.
We may now write $v$ uniquely as a sum
\[
	v
	= 
	\sum_{k,\ell, k',\ell'} 
	\alpha_{k,\ell;k',\ell'}^{}
	w^{(1)}_{k,\ell}\tensor w^{(2)}_{k',\ell'},
\]
for some $\alpha_{k,\ell;k',\ell'}\in \F_q$,
and the claim is that $\alpha_{k,\ell;k',\ell'} = 0$
whenever $k+k' > i$. 
Since $(I_1 \tensor 1) \cdot v \subset J_{i-1}(\xi_1\tensor \xi_2)$,
writing $v$ as 
\begin{equation}
\label{eq:vexp} 
	v
	= 
	\sum_{k',\ell'} 
	\bigg(
		\sum_{k,\ell}
		\alpha_{k,\ell;k',\ell'}^{}
		w^{(1)}_{k,\ell}
	\bigg)
	\tensor 
	w^{(2)}_{k',\ell'},
\end{equation}
and taking into account
the inductive assumption,
we see that the coefficient of $w^{(2)}_{k',\ell'}$ in \eqref{eq:vexp}
must be in 
$J_{\max(0,i-k')}(\xi_1)$.
Therefore 
$\alpha_{k,\ell;k',\ell'} = 0$
when both $k+k' > i$ and $k>0$. Similarly, 
$\alpha_{k,\ell;k',\ell'} = 0$
when both $k+k' > i$ and $k'>0$, proving the claim.
\end{proof}

Lemma~\ref{lm:jandtensor} implies in particular
\begin{lemma}\label{J1wedge}
Suppose $\rho$ and $\eta$ are representations of $p$-groups $G,G'$ 
such that the fixed-point subspaces of $\rho$ and $\eta$
are one-dimensional.  Then
\[ J_1(\rho\otimes\eta) = J_1(\rho)\vee J_1(\eta) \]
as representations of $G\times G'$. \qed
\end{lemma}

\begin{cor}\label{tensorproductformula}
If $\rho$ and $\eta$ are representations of elementary abelian groups $G,G'$, then
\[ \chi_\alpha(\rho\otimes\eta)=-\sum \chi_{\alpha_{(1)}}(\rho)\times\chi_{\alpha_{(2)}}(\eta), \]
where
\[ \Delta(\alpha)=\sum \alpha_{(1)}\otimes\alpha_{(2)}. \]
\end{cor}

\begin{proof}
Both $\rho$ and $\eta$ must of course have at least a one-dimensional fixed-point space.  And if either of them has more than one fixed line, then so does the tensor product. Consequently, by Lemma~\ref{twofixedlines}, both sides of the equation vanish.
So we assume that $\dim J_0(\rho)=\dim J_0(\eta)=1$.  Now the result is immediate from Lemma~\ref{J1wedge}, Theorem~\ref{J1} and Theorem~\ref{wedgeformula}.
\end{proof}

Lastly, let us restrict our attention for a moment to the case $r=1$.  In that case, the calculation of $\chi_\alpha(\rho_A)$ in section~\ref{basicrep} actually suffices to compute $\chi_\alpha(\bullet)$ for all representations of elementary abelian $p$-groups.  Indeed:
\begin{thm}
Let $\xi:A\to GL_N\F_p$ be a representation of an elementary abelian $p$-group $A$.  Then
\[ \chi_\alpha(\xi)=\pi^*\chi_\alpha(\rho_{A/B}) \]
if $\dim J_0(\xi)=1$, and zero otherwise.
Here $B\leq A$ is the kernel of the action of $A$ on $J_1(\xi)$ and $\pi:A\to A/B$ is the projection.
\end{thm}

\begin{proof}
Assume $J_0(\xi)$ is one-dimensional; otherwise $\chi_\alpha(\xi)=0$ 
by Theorem~\ref{twofixedlines}. By Theorem~\ref{J1}, the left hand side 
depends only on $J_1(\xi)$.  Since the same is true for the right hand side, 
we may assume that $\xi=J_1(\xi)$.
Then $B$ is the kernel of $\xi$, so
\[ \chi_\alpha(\xi)=\pi^*\chi_\alpha(\tilde\xi), \]
where $\tilde\xi$ is the corresponding representation of the quotient $A/B$,
a faithful representation.  Hence we have reduced to the case of a representation $\xi$ of $A$ which is faithful and satisfies 
$J_1(\xi)=\xi$ and $\dim J_0(\xi) = 1$.  In this case, the claim reads
\[ \chi_\alpha(\xi)=\chi_\alpha(\rho_A), \]
so the proof is completed by verifying that
under these conditions \[ \xi\iso\rho_A,\]
which we will now do. Write $\xi$ as an internal direct sum 
$\xi = J_0(\xi) \oplus \xi_1$ of vector spaces. 
Then, since the vectors in $J_0(\xi)$
are fixed in the $A$-action, and $\xi = J_1(\xi)$, the action of $A$
on $\xi$ is given for 
$a\in A, (v_0,v_1)\in J_0(\xi)\oplus \xi_1$ by the formula
\[
	a\cdot (v_0,v_1) = (v_0 + \omega(a,v_1), v_1)
\] 
for some bilinear map $\omega\colon A \times \xi_1 \to J_0(\xi)$.
Pick an isomorphism $\theta\colon J_0(\xi) \isom \F_p$.
Since the action of $A$ on $\xi$ is faithful, and 
nonzero vectors in $\xi_1$ are not fixed points,
the pairing $\theta\omega$ is perfect.
Now the map
\[
	\xi = J_0(\xi)\oplus \xi_1 \longto \rho_A = \F_p \oplus A^*,
	\quad
	(v_0,v_1) \longmapsto (\theta(v_0), \theta\omega(-,v_1))
\]
is an isomorphism of $A$-representations.
\end{proof}

\section{Cohomology of $GL_N\F_q$}\label{largerep}
In this section, we aim to prove nontriviality of 
$\chi_\alpha^{(N)}\in H^*(GL_N\F_q)$ for $N$ as large as possible.  
We do so by constructing a representation whose characteristic classes
will coincide with those studied in section~\ref{nontriviality} but whose
dimension is much larger:

\begin{lemma}\label{bigrep}
For each $n\geq1$, there is a representation $\xi_A$ of 
the elementary abelian group $A=\F_q^n$ such that
\[ \dim(\xi_A)=p^n \+ J_1(\xi_A)\iso\rho_A. \]
In the case $r=1$, $\xi_A$ may be taken to be the regular representation of $A$.
\end{lemma}
\begin{proof} 
In light of Lemma~\ref{J1wedge} and the fact that $\rho_V\vee\rho_W=\rho_{V\oplus W}$, we may immediately reduce to the case $n=1$ by defining
\[ \xi_{\F_q^n}=\underbrace{\xi_{\F_q}\otimes\cdots\otimes\xi_{\F_q}}_n. \]
In the case $n=1$, we may take
\[ \xi_{\F_q}=\sym^{p-1}(\F_q^2), \]
a $p$-dimensional representation of 
\[
	\F_q = \begin{bmatrix}1&*\\&1\end{bmatrix} \leq GL_2\F_q.
\]
In terms of the basis $\F_q^2=\langle e_1,e_2\rangle$,
\[ \xi_{\F_q}=\langle e_1^{p-1},e_1^{p-2}e_2,\dots,e_2^{p-1}\rangle, \]
and we have
\[ J_i(\xi_{\F_q})=\langle e_1^{p-1-j}e_2^j \mid j\leq i\rangle. \]
In particular, the action on $J_1(\xi_{\F_q})$ is the standard one
\[ \rho_{\F_q}:a\longmapsto\begin{bmatrix}1&a\\&1\end{bmatrix}. \]

Now assume $r=1$.  We will check that $\xi_{\F_p^n}$ 
as defined above is isomorphic to the regular representation of 
$\F_p^n$.
Since both representations split up as an $n$-fold tensor product, 
it suffices to prove this in the case $n=1$.  
In that case, we need to check that the two endomorphism of order 
$p$ have the same Jordan structure; 
notice that such an endomorphism satisfies the polynomial
$x^p-1 = (x-1)^p$, so all its eigenvalues are equal to 1,
and it will have a Jordan form over $\F_p$ without needing to 
extend the field.
Now observe that both endomorphisms have a single fixed line, hence a single block.
\end{proof}

Using the characteristic classes of the representation in Lemma~\ref{bigrep},
we can now produce many nonzero cohomology classes on the general linear groups:
\begin{thm}\label{nonvanishing}
Fix $n\geq1$.  For all \[ 2\leq N\leq p^n, \]
\begin{enumerate}[(i)]
\item\label{thmpart:nonvanishingnilpotent} 
The class 
\[
	\chi^{(N)}_\alpha \in H^\ast(GL_N\F_q;\,\F_q)
	\quad\text{with}\quad
	\alpha = (y_0\cdots y_{r-1})^{p^n-1}
\]
is a non-nilpotent element of degree $2r(p^n-1)$ for $p$ odd, 
or degree $r(2^n-1)$ for $p=2$.
\item 
For $p$ odd, the class
\[
	\chi^{(N)}_\alpha \in H^\ast(GL_N\F_q;\,\F_q)
	\quad\text{with}\quad
	\alpha = x_0\cdots x_{r-1}(y_0\cdots y_{r-1})^{p^n-p^{n-1}-1}
\]
is a nonzero element of degree $r(2p^n-2p^{n-1}-1)$.
\end{enumerate}
\end{thm}
\begin{proof}
In the special case $N=p^n$, we need only combine Lemma~\ref{bigrep} with 
Theorem~\ref{J1} and Proposition~\ref{nonvanishingforbasicreps}. 
The general case now follows from Corollary~\ref{cor:parindclasses}.
\end{proof}

\begin{remark}
The universal classes $\chi_\alpha^{(n)}$ are not the only nonzero
cohomology classes on $GL_N\F_q$ which can be detected
using our characteristic classes: there are also ``dual classes''
\[ \chi_\alpha((\F_q^n)^*)=\chi_\alpha(T:GL_n\to GL_n)=T^*(\chi_\alpha^{(n)}), \]
where \[ T:GL_n\F_q\to GL_n\F_q \]
is the inverse-transpose automorphism $A\mapsto (A\inv)^t$.  These dual classes are in general distinct from the $\chi_\alpha^{(n)}$ classes.  Indeed, for $n>2$, let $V=\F_q^{n-1}$ and consider the dual $\nu=\rho_V^*$ of the basic representation.  Then
\[ \nu^*\chi_\alpha^{(n)}=\chi_\alpha(\nu)=0 \]
for all $\alpha$ by Theorem~\ref{twofixedlines}, while
\[ \nu^*T^*\chi_\alpha^{(n)}=\chi_\alpha(\nu^*)=\chi_\alpha(\rho_V) \]
which is often nonzero, as we saw in section~\ref{nontriviality}.
\end{remark}

In the case $r=1$, we have from Proposition~\ref{primecasenonvanishing} (along with Theorem~\ref{J1}, Lemma~\ref{bigrep}, and Corollary~\ref{cor:parindclasses}):
\begin{thm}\label{r1nonvanishing}
Fix $n\geq1$.  For all \[ 2\leq N\leq p^n, \]
\begin{enumerate}[(i)]
\item If $p=2$, the class 
\[
	\chi_{y^d}^{(N)} \in H^d(GL_N\F_2;\,\F_2)
\]
is a non-nilpotent element whenever the sum of the binary digits of $d$ is at least $n$.
\item If $p$ is odd, the class
\[
	\chi_{y^d}^{(N)} \in H^{2d}(GL_N\F_p;\,\F_p)
\] 
is a non-nilpotent element whenever the sum of the $p$-ary digits of $d$ is $k(p-1)$ for some $k\geq n$.
\item If $p$ is odd, the class
\[
	\chi_{xy^d}^{(N)} \in H^{2d+1}(GL_N\F_p;\,\F_p)
\]
is a nonzero element whenever the sum of the $p$-ary digits of $d$ is $k(p-1)-1$ for some $k\geq n$.
\qed
\end{enumerate}
\end{thm}

\section{Dickson invariants}\label{Dickson}

For this section, we restrict to the case $r=1$ and assume $\F=\F_p$.
We give an alternative description of the  classes
\[ \chi_{y^k}(\rho_{\F_p^n})\in H^{\ast}(\F_p^n). \]
By equation~\eqref{multinomialformula}, these classes 
belong to the polynomial subalgebra 
\[
	\F_p[{\F_p^n}^*] \subset H^\ast(\F_p^n),
\]
where we interpret ${\F_p^n}^*$ as $H^1(\F_p)$ if $p=2$ or as
the image of the Bockstein map $\beta:H^1(\F_p^n)\to H^2(\F_p^n)$
if $p$ is odd.
Because the image of $\rho_{\F_p^n}$ in $GL_{n+1}\F_p$ is normalized by a copy of 
$GL_n\F_p$, these classes live in the invariant subring
\[ 
	\F_p[{\F_p^n}^*]^{GL_n\F_p} = \F_p[ D_{p^n-p^{n-1}},\ldots, D_{p^n-1}]
\]
where the elements $D_{p^n-p^i}$ are the Dickson invariants
of ${\F_p^n}^*$. Our aim is to express the classes 
$\chi_{y^k}(\rho_{\F_p^n})$ as polynomials in the Dickson invariants.

We begin with an alternative, coordinate-independent description
of the classes $\chi_{y^k}(\rho_{\F_p^n})$. 
So far, we have defined these classes when $k$ is a multiple of $p-1$.
It is convenient to extend the definition by setting 
$\chi_{y^k} = 0$ when $p-1$ does not divide $k$. Then:
\begin{lemma}
For every $k >0$, we have
\[ \chi_{y^k}(\rho_{\F_p^n})=-\sum_{z\in{{\F_p^n}^*}}z^k. \]
\end{lemma}
\begin{proof}
By expanding $z=c_1z_1+\cdots+c_nz_n$ according to the dual 
$\{z_1,\dots,z_n\}$
of the standard basis, we obtain
\begin{align*}
	\sum_{z\in{{\F_p^n}^*}}z^k &=
	\sum_{c_1,\cdots,c_n\in\F_p} 
	(c_1z_1+\cdots + c_nz_n)^k
	\\
	&=\sum_{i_1+\cdots+i_n=k}\binom{k}{i_1,\dots,i_n} 
	\left(\sum_{c\in\F_p}c^{i_1}z^{i_1}\right)\times\cdots
	\times\left(\sum_{c\in\F_p}c^{i_n}z^{i_n}\right)
	\\
	&=(-1)^n\sum_{\substack{i_1+\cdots+i_n=k \\ (p-1)|i_j>0\ \forall j}}
	\binom{k}{i_1,\dots,i_n} z^{i_1}\times\cdots\times z^{i_n},
\end{align*}
which agrees with the right hand side of 
equation~\eqref{multinomialformula} (in the case $q=p, A=0, B=k$)
except for a factor of $-1$.
The last equality holds because
$\sum_{c\in\F_p}c^i$ is equal to $-1$
if $(p-1)$ divides $i$ and $i\neq0$, and to $0$ otherwise.
\end{proof}

In other words, the classes $-\chi_{y^k}(\rho_{\F_p^n})$
can be viewed as the 
power sum symmetric polynomials
evaluated on the 
elements of ${\F_p^n}^*$.  Since the Dickson polynomials are the elementary symmetric 
polynomials evaluated on the same elements, Newton's identity gives a relationship 
between the two.
To express it conveniently, write total classes (as elements of the degree-wise completed cohomology)
\[ D=\sum_{i=0}^n D_{p^n-p^i} = \prod_{z\in {\F_p^n}^*}(1+z)
=\sum_{k=0}^{p^n}\sigma_k({\F_p^n}^*), \]
\[ A=\sum_{k>0}(-1)^k\chi_{y^k}(\rho_{\F_p^n})  
=\sum_{k>0}(-1)^{k-1}p_k({\F_p^n}^*). \]
Here $\sigma_k$ is the $k$-th elementary symmetric polynomial, and
$p_k(x_1,\dots,x_\ell)=\sum_i x_i^k$ is the $k$-th power sum polynomial.

From Newton's identity
\[ \left(\sum_{i=0}^\infty \sigma_i\right)\left(\sum_{i=1}^\infty(-1)^{i-1} p_i\right)=\sum_{k=1}^\infty k\sigma_k, \]
we get
\[ DA=\sum_{k>0}kD_k = -D_{p^n-1}. \]
The last equality holds because, for each $k$, either $p|k$ or $D_k=0$ or $k=p^n-1$.
Consequently,
\begin{thm} 
We have
\[ 
	\pushQED{\qed} 
	A=-D_{p^n-1}\cdot D\inv. 
	\qedhere 
	\popQED
\]
\end{thm}

We can write the lowest terms more explicitly:
\[ A = -D_{p^n-1}(1-D_{p^n-p^{n-1}}-\cdots-D_{p^n-1}+\textrm{higher terms}) \]
using the calculation
\[
	D^{-1} 
	= (1+\tilde{D})^{-1} 
	= 1-\tilde{D} + \tilde{D}^2 -\cdots\
	= 1-\tilde{D} + \textrm{higher terms}
\]
where $\tilde{D} = D_{p^n-p^{n-1}}+\cdots+D_{p^n-1}$.
Consequently, we have
\[ \chi_{y^{2p^n-p^i-1}}(\rho_{\F_p^n})  =\pm D_{p^n-1}D_{p^n-p^i},
\quad 0\leq i \leq n, \]
and these are the only $\chi_{y^k}(\rho_{\F_p^n})$'s which are nonzero
for $k\leq 2(p^n-1)$.
Because the Dickson polynomials are algebraically independent, it follows that 
the $n$ classes
\[ 
	D_{p^n-1}
	\quad\text{and}\quad
	D_{p^n-1}D_{p^n-p^i},\quad 1\leq i \leq n-1,
\]
are algebraically independent.  We can deduce:
\begin{thm}
\label{thm:algindependence}
Suppose $n+1\leq N\leq p^n$.  Then the universal classes
\[ \chi_{y^{2p^n-p^i-1}}^{(N)}\in H^*(GL_N\F_p;\,\F_p),\quad 1\leq i\leq n \]
are algebraically independent.  In particular, our 
characteristic classes $\chi_\bullet^{(N)}$
generate a subring of Krull dimension at least $n$.
\end{thm}
\begin{proof}
In light of the above discussion and Theorem~\ref{J1}, it suffices to produce a 
representation $\xi$ of $\F_p^n$ with $\dim(\xi)=N$ and $J_1(\xi)=\rho_{\F_p^n}$.
In the case $N=p^n$, we have the representation $\xi=\xi_{\F_p^n}$ of Lemma~\ref{bigrep}.
In the general case, any $N$-dimensional subrepresentation $\xi$ with
\[ \rho_{\F_p^n}=J_1(\xi_{\F_p^n})\leq\xi\leq\xi_{\F_p^n} \]
will suffice.  Such subrepresentations exist with every dimension between
$n+1=\dim\rho_{\F_p^n}$ and $p^n=\dim\xi_{\F_p^n}$;
they can be constructed by upper-triangularizing the quotient
$\xi_{\F_p^n}/\rho_{\F_p^n}$.
\end{proof}

\section{Applications to $\aut(F_n)$ and $GL_n\Z$}
\label{sec:othergroups}
We assume $r=1$ and $\F=\F_p$ throughout the section.
In this section, our aim is to demonstrate the usefulness
of our characteristic classes for studying groups other 
than general linear groups over finite fields by 
presenting applications of our 
computations to the homology and cohomology
of automorphism groups of free groups and 
general linear groups over the integers.  
The applications are of a broadly 
similar type to those we have presented
for the cohomology of general linear groups over finite fields:
we construct large families of explicit nontrivial elements in the 
homology and cohomology of $\aut(F_{p^n})$ and $GL_{p^n}\Z$ 
living in the  unstable range where the homology and cohomology 
groups remain mysterious. In the case of homology, 
the classes are not only nontrivial, but in fact indecomposable
in the rings $H_\ast(\bigsqcup_{n\geq 0} B\aut(F_n))$ and 
$H_\ast(\bigsqcup_{n\geq 0} BGL_n\Z)$. (We will indicate the ring 
structures below.)

Our starting point is the observation that the 
regular representation $\rho_\mathrm{reg}$ of $\F_p^n$
factors as the composite
\begin{equation}
\label{eq:regrepfactors}
	\F_p^n 
	\xto{\ c\ }
	\Sigma_{p^n} 
	\xto{\ i\ }
	\aut(F_{p^n})
	\xto{\ \pi_\mathrm{ab}\ }
	GL_{p^n}\Z
	\xto{\ \rho_\mathrm{can}\ }
	GL_{p^n}\F_p
\end{equation}
where $c$ is the Cayley embedding, $i$ is the embedding sending
a permutation to the corresponding automorphism of the free group 
$F_{p^n}$ given by permuting generators,
$\pi_\mathrm{ab}$ is given by abelianization,
and $\rho_\mathrm{can}$ is the canonical representation of $GL_{p^n}\Z$
on $\F_{p^n}$  given by reduction mod $p$.
By Lemma~\ref{bigrep} and Theorem~\ref{J1},
in the case $r=1$, the characteristic classes of 
the regular representation $\rho_\mathrm{reg}$ of $\F_p^n$
agree with those of the basic representation $\rho_{\F_p^n}$:
\begin{equation}
\label{eq:regandbasic}
\chi_\alpha(\rho_\mathrm{reg}) = \chi_\alpha(\rho_{\F_p^n})
\end{equation}
for all $\alpha$. Thus we obtain
\begin{thm}
\label{thm:glzautfn}
Suppose $\alpha\in H^\ast(GL_2\F_p;\,\F_p)$ is any class for which 
$\chi_\alpha(\rho_{\F_p^n}) \neq 0$. Then 
\[
	\chi_\alpha(\rho_\mathrm{can}) 
	\neq 
	0 \in H^\ast(GL_{p^n}\Z;\,\F_p)
\]
and
\[
	\chi_\alpha( \pi_\mathrm{ab}^\ast( \rho^{}_\mathrm{can}))
	\neq 
	0 \in H^\ast(\aut(F_{p^n});\,\F_p). 
\]
Moreover, if $\chi_\alpha(\rho_{\F_p^n})$ is non-nilpotent,
so are $\chi_\alpha(\rho_\mathrm{can})$ and 
$\chi_\alpha( \pi_\mathrm{ab}^\ast( \rho^{}_\mathrm{can}))$. 
\qed
\end{thm}
\noindent
We remind the reader that 
the question of when $\chi_\alpha(\rho_{\F_p^n})$ is nonzero 
was given a complete answer in Proposition~\ref{primecasenonvanishing}.

We now turn to the application to homology groups.
The homology groups of each of the spaces
\[
\bigsqcup_{n\geq 0} B\Sigma_n,
\quad
\bigsqcup_{n\geq 0} B\aut(F_n),
\quad
\bigsqcup_{n\geq 0} BGL_n\Z
\quad
\text{and}
\quad
\bigsqcup_{n\geq 0} BGL_n\F_p
\]
are highly structured.
First, in each case the homology carries a ring structure.
These structures are 
induced by
disjoint unions of sets (yielding homomorphisms 
$\Sigma_n \times \Sigma_m \to \Sigma_{n+m}$);
by free products of free groups
(yielding homomorphisms 
$\aut(F_n) \times \aut(F_m) \to \aut(F_{n+m})$);
and by direct sums of free modules
(yielding the block-sum homomorphisms on general linear groups).
Second, with the apparent exception of automorphism groups of 
free groups, in each case the homology groups carry an additional
product, which we denote by $\circ$.
These products are induced
by direct products of sets (yielding 
homomorphisms $\Sigma_n \times \Sigma_m \to \Sigma_{nm}$)
and tensor products of free modules
(yielding homomorphisms 
$GL_nR \times GL_mR \to GL_{nm}R$).
The maps induced by the analogues
\[
	\Sigma_{n} 
	\xto{\ i\ }
	\aut(F_{n})
	\xto{\ \pi_\mathrm{ab}\ }
	GL_{n}\Z
	\xto{\ \rho_\mathrm{can}\ }
	GL_{n}\F_p,
	\quad
	n\geq 0 
\]
of the maps appearing in \eqref{eq:regrepfactors} 
are compatible with all this structure in the sense that 
all give ring homomorphisms, and both
$(\pi_\mathrm{ab}i)_\ast$ and $(\rho_\mathrm{can})_\ast$
preserve the $\circ$-product.
For each $k\geq 0$, let $z_k \in H_\ast(\F_p)$ 
be the dual of $y^k\in H^\ast(\F_p)$, and write
\begin{align*}
E_k &= c_\ast (z_k) \in H_\ast(\Sigma_p)
\\
E^{\Z}_k &= (\pi_\mathrm{ab} i)_\ast(E_k) \in H_\ast(GL_p\Z)
\\
E^{\F_p}_k &= (\rho_\mathrm{can})_\ast (E^{\Z}_k) \in H_\ast(GL_p\F_p),
\end{align*}
where $c\colon \F_p \to \Sigma_p$ is the Cayley embedding.
With this notation, we have:
\begin{thm}
\label{thm:indecomposables}
Suppose $B_1,\ldots,B_n \in \Z$ are positive multiples of $p-1$ such that 
there is no carry when $B_1,\ldots,B_n$ are added together in base $p$.
Let $B= B_1 + \cdots + B_n$.
Then the following elements are indecomposable in their respective rings:
\begin{enumerate}[(i)]
\item $i_\ast (E_{B_1} \circ \cdots\circ E_{B_n})
	 \in 
	 H_{2B}(\aut(F_{p^n});\,\F_p)$ 
	 in 
	 $H_\ast (\bigsqcup_{k\geq 0} B\aut(F_k);\,\F_p)$
\item $E^\Z_{B_1} \circ \cdots \circ E^\Z_{B_n} 
	\in 
	H_{2B}(GL_{p^n}\Z;\,\F_p)$ 
	in 
	$H_\ast (\bigsqcup_{k\geq 0} BGL_k\Z;\,\F_p)$
\item $E^{\F_p}_{B_1} \circ \cdots \circ E^{\F_p}_{B_n} 
	\in 
	H_{2B}(GL_{p^n}\F_p;\,\F_p)$ 
	in 
	$H_\ast (\bigsqcup_{k\geq 0} BGL_k{\F_p};\,\F_p)$.
\end{enumerate}
(For $p=2$, replace $2B$ by $B$.)
\end{thm}
\begin{proof}
The first two elements map to the third under the ring homomorphisms
$(\pi_{\mathrm{ab}})_\ast$
and
$(\rho_\mathrm{can})_\ast$,
so it suffices to prove the third claim.
Observe that the Cayley embedding $\F_p^n \to \Sigma_{p^n}$
factors up to conjugacy as 
\[
	\F_p^n \xto{\ c^{\times n}\ } \Sigma_p^n \longto \Sigma_{p^n} 
\]
where the latter map induces the iterated $\circ$-product.
Thus, by the factorization \eqref{eq:regrepfactors} of the
regular representation, we have
\[
	E^{\F_p}_{B_1} \circ \cdots \circ E^{\F_p}_{B_n }
	= 
	(\rho_\mathrm{reg})_\ast (z_{B_1}\times \cdots \times z_{B_n}).
\]
We obtain
\begin{align*}
\langle 
\chi_{y^B}^{(p^B)}, E^{\F_p}_{B_1} \circ \cdots \circ E^{\F_p}_{B_n} \rangle 
&=
\langle 
\chi_{y^B}(\rho_\mathrm{reg}),
z_{B_1}\times \cdots \times z_{B_n}
\rangle 
\\
&=
\langle 
\chi_{y^B}(\rho_{\F_p^n}),
z_{B_1}\times \cdots \times z_{B_n}
\rangle 
\\
&=
(-1)^{n-1} \binom{B}{B_1,\dots,B_n}
\\
&\neq 0 \in \F_p.
\end{align*}
Here the second equality follows from equation \eqref{eq:regandbasic},
the third from equation \eqref{multinomialformula}, and the 
final inequality from the choice of $B_1,\ldots,B_n$.
The claim now follows from Remark~\ref{rk:primitives}.
\end{proof}

\begin{remark}
In the case $p \neq 2$, we can obtain
further indecomposables in the aforementioned rings
by a similar argument 
by taking into account the exterior class in $H^\ast(\F_p)$.
We leave the formulation of the result to the reader.
\end{remark}

\begin{remark}
Analogues of Theorems~\ref{thm:glzautfn} and \ref{thm:indecomposables}
hold, by the same arguments,
for many other interesting groups through which 
the regular representation factors.
An example is $GL_{p^n}R$ for $R$ any ring surjecting onto $\F_p$,
such as $R = \Z/p^k$ or $R = \Z_p$.
\end{remark}

\begin{vista}
In addition to the classes $\chi_\alpha$ for 
$\alpha\in H^\ast(GL_2 \F_q)$,
one could define similar characteristic classes
for $\alpha \in H^\ast(GL_n \F_q)$ for any $n$
by setting, for a $G$-representation $\rho$ of
dimension $N \geq n$,
\[
	\chi_\alpha(\rho) = \rho^\ast \chi^{(N)}_\alpha \in H^\ast (G)
\]
where the class $\chi^{(N)}_\alpha \in H^\ast(GL_N\F_q)$
is defined in terms of the parabolic induction map $\Phi_{n,N}$
as $\chi^{(N)}_\alpha = \Phi_{n,N}(\alpha)$.
Compare with Definition~\ref{def:chiclassdef1} 
and Remark~\ref{rk:chiparindconn}.
What makes the case $n=2$ particularly interesting is 
the computability of $H^\ast(GL_2\F_q)$
and the existence of the coproduct structure 
appearing in the wedge sum formula of Theorem~\ref{wedgeformula}.
In view of the compatibility of the parabolic induction maps
under composition (Proposition~\ref{prop:parindcompat}),
the classes $\chi_\alpha$ for $\alpha$ of the form
$\chi_\beta^{(n)}$ 
will not yield any interesting new information, 
but it is possible that the classes $\chi_\alpha$
for $\alpha$, say, a product of some previously constructed classes
will. 
We leave the exploration of this idea for future work.  
\end{vista}

\section*{Acknowledgements}

The authors would like to thank Nathalie Wahl and Jesper Grodal
for valuable comments and advice and the anonymous referee
for useful suggestions and corrections.
The second author was supported by the Danish National Research Foundation
through the Centre for Symmetry and Deformation (DNRF92).

\bibliographystyle{amsalpha}
\bibliography{modular-characteristic-classes}

\end{document}